\PassOptionsToPackage{numbers}{natbib}
\documentclass[graybox,natbib,envcountsame]{svmult}

\usepackage{mathptmx}
\usepackage{helvet}
\usepackage{courier}
\usepackage{makeidx}
\usepackage{graphicx}
\usepackage{multicol}
\usepackage[bottom]{footmisc}
\usepackage{float}

\usepackage{amsmath}   
\usepackage{amssymb} % Note: amssymb loads amsfonts
\usepackage[utf8]{inputenc}
\usepackage{tikz}
\usetikzlibrary{%
  cd, % Commutative diagrams
  matrix % Matrix of math nodes, etc.
}

\usepackage{mathrsfs}
\usepackage[all]{xy}

\newcommand{\R}{\mathbb{R}}

\DeclareMathOperator{\isom}{Isom}
\newcommand{\void}{\textnormal{\O}}
\spnewtheorem*{AAT}{Abel's Addition Theorem}{\bfseries}{\itshape}

\smartqed
\usepackage{etoolbox}
\patchcmd{\qed}{\ifmmode\qedsymbol}{\ifmmode\the\qedsymbol}{}{\foobar}
\patchcmd{\qed}{\hfil\qedsymbol}{\hfil\the\qedsymbol}{}{\foobar}
\qedsymbol{\proofSymbol}

\newcommand{\qedhere}{\tag*{\the\qedsymbol}}

\usepackage{hyperref}
\def\C{\mathbb C} 
\def\D{\mathbb D} 
\def\N{\mathbb N}

\begin{document}

\title*{Dennis Sullivan's Work on Dynamics}

%
% Author(s)
%
\author{Edson de Faria and Sebastian van Strien}
\institute{E. de Faria \at Instituto de Matematica e Estatistica, Universidade de Sao Paulo, \email{edson@ime.usp.br} \and S. van Strien \at Department of Mathematics, Imperial College London, \email{s.van-strien@imperial.ac.uk}}

\maketitle

%\chapter{Dennis Sullivan's Work on Dynamics}

Before going into Dennis Sullivan's work on dynamics, we would like to share some of our reminiscences on the  remarkable 
way in which he influenced a huge number of mathematicians, including the two of us.  Both while at IHES
and at CUNY, Dennis had an office which came with an anteroom.  
Our impression is that he would spend most of his time in this anteroom,  
 talking about mathematics with whoever he had invited or whoever was around. 
 Quite often while listening to somebody, he would end up giving a new spin or a 
 new interpretation  to what they had been saying. Similarly, he would 
 explain what he was working on, trying out new ideas, and also often explaining 
 results of others.  Spending time with him was always an incredible experience. 

In this spirit, Dennis explained much of his work on renormalisation to Welington de Melo. 
In turn, Welington would then try to explain what he had heard and learned to SvS. 
When he could not convince SvS of some argument, Welington would 
go back to Dennis and this process would repeat again, sometimes many times. This is how 
the final chapter of the book {\em One dimensional dynamics} of SvS with Welington de Melo
  came into being, see \cite{MR1239171}.  This chapter contains a full exposition of Dennis' remarkable 
renormalisation theory, arguably the only place in which it  was published with full details. 
 
At the Graduate Center of CUNY, Dennis coordinated the \emph{Einstein Chair Seminar}, more informally known as the Sullivan seminar, bringing in speakers from all over the world. The seminar ran once a week with talks in the afternoon, but the invited speakers usually came in the morning, and intense mathematical discussions ensued through lunch, and oftentimes even after the speaker's talk, following a short break for tea.  
During the talks, Dennis often asked questions not necessarily  to 
know the answer himself,  but because he knew that somebody else in the audience 
would find the answer helpful. In this way, Dennis  took  on the role
of introducing two people to each other. His presence in the audience would usually  
make a talk much more accessible and interesting. His questions  would often 
clarify connections that would have remained implicit otherwise. 

When Dennis invented or learned about a new mathematical idea, he would push 
this idea to the limit. For him it was very important to understand what this idea
would give, and equally important to find out what the limitations of this idea might be. 
Moreover, whenever possible, he liked to associate names to arguments
 such as the  {\em dollar argument}, {\em smallest interval argument} or 
the {\em non-coiling argument} in order to synthesise a complex proof
into its core ideas. 

Often he mentioned that to understand a proof properly, you should treat it like a three dimensional object. 
You should not only look at it from one side, but from all sides. So in this sense, in his view, 
a proof was about mathematical understanding rather than about {\lq}killing{\rq} a theorem. 

Indeed,  Dennis' interest in a result might not necessarily be in the power of the statement per se, but 
in the tools that are used in the proof of this result.  Once he understands the tools and ideas, then he probably can 
recover the statement of the results by himself.

\medskip 

So let us turn to the  field of dynamical systems. 
Dennis Sullivan always had a keen interest in the field of dynamical systems, 
and already in the 1970's published several high impact papers 
in this area, many of them remarkably short.  Let us first highlight a few of his papers
on smooth dynamical systems, and then his  groundbreaking papers on  Kleinian groups and holomorphic dynamics.

\section{Smooth dynamics} 

Although his early papers on dynamics treat a diverse range of problems, they  all have an overarching theme: 
what smoothness (or other) structures are compatible with a particular dynamical setting.

\subsection{From topology to dynamics} 

In the mid-seventies, Dennis started to become more and more interested in dynamics, transitioning from the pure study of structures on manifolds to the study of dynamical objects such as flows and, more generally, foliations on manifolds. 
One of his first striking results in this direction is the following.

\begin{theorem}[Counterexample to the periodic orbit conjecture, see 
\cite{MR402824}, \cite{MR501022}] 
There exists a flow on a compact five-dimensional manifold all of whose orbits are periodic, and yet the lenghts of such orbits are not bounded.  
\end{theorem} 

This theorem{\footnote{This theorem is also discussed in McMullen's beautiful talk on Dennis' work at MSRI in the spring of 2022, see \url{https://vimeo.com/702914316?embedded=true&source=vimeo_logo&owner=106107493}}}
 was announced in \cite{MR402824}, and a more detailed argument given in \cite{MR501022}. Dennis' topological-geometric construction in the latter paper yields a flow on a smooth $5$-dimensional manifold $M$ which is Lipschitz, but he states that he sees no reason why the example could not be made smooth. 
In an addendum at the end of the paper, Dennis briefly explains an idea due to Kuiper that results in an example in which the flow is $C^\infty$. He also explains a beautiful construction due to Thurston which yields a real-analytic flow with the desired properties  on the manifold $M=\mathbb{R}^5/(\Gamma\times \mathbb{Z}\times \mathbb{Z})$, where $\Gamma$ is the group
\[
\Gamma = \left\{ \left( \begin{matrix}
1 & x& y\\
0&1&z\\
0&0&1\\
\end{matrix} \right)\;:\; x,y,z\in \mathbb{Z} \right\}\ .
\]
In other words, $\Gamma$ is the so-called \emph{discrete Heisenberg group}. 

\subsection{Rigidity in smooth dynamics} 

The following paper, joint with Shub,  dates back to 1985. 
\begin{theorem}[Expanding maps of the circle, see \cite{MR796755}]
Two $C^r$, $r\ge 2$, expanding maps of the circle which are absolutely continuously conjugate are $C^r$ conjugate. %Here $f$ and $g:S^1\to S^1$ are expanding if they stretch tangent vectors in some metric, and a conjugacy is a
%%homeomorphism $h:S^1\to S^1$ such that $fh=hg$. 
\end{theorem} 

The proof of this result is short and starts by invoking well-known results to reduce the problem 
to the setting where both maps preserve the Lebesgue measure. Using the assumption that the maps
are $C^2$ and expanding makes it possible to consider an iterate of the maps taking a small
interval to big scale with bounded distortion. This then
implies that the assumption that the conjugacy $h$ is absolutely continuous gives that $h$ is, in fact,  Lipschitz. 
Going on from there, one obtains that $h$ is $C^r$. 

In some  broader sense the main idea of this paper was the starting point for quite a lot of later research. 
Indeed, the pullback argument that Dennis introduced in the field of holomorphic dynamics
is somewhat similar in spirit. Quite a few other papers followed on from this work. For example,
\begin{itemize}
\item if the multipliers of corresponding periodic points of two topologically conjugate unimodal interval are equal, 
then the conjugacy is smooth, see \cite{MR1677736, MR2229794} (there are corresponding results in higher dimensions); 
\item if a conjugacy between two interval maps is smooth at some point, then it is smooth everywhere, see 
\cite{MR3174743};
\item there are quite a few very interesting related results for group actions of circle maps, see for example \cite{MR2358052}; 
\item  when studying the smoothness of a conjugacy between two maps on defined on Cantor 
sets,  Sullivan introduced the notion of a  scaling function on a Cantor set, see \cite{MR974329} and also \cite{MR2226493}. 
\end{itemize}

\medskip

Another paper, joint with Norton,  on rigidity in dynamics is concerned with Denjoy examples of $C^1$ torus diffeomorphisms
$T^2\to T^2$ which are isotopic to the identity. These are maps $f$ which are  
topologically semi-conjugate to a minimal translation 
on the torus, i.e. $hf=Rh$, where $h$ is a continuous map of $T^2$ onto itself, homotopic to the identity, such that the set of $x\in T^2$ for which the cardinality of $h^{-1}(x)$ is greater than 1 is nonempty and countable. Then the interior of any $h^{-1}(x)$, if nonempty, is a 
wandering domain for $f$.

\begin{theorem}[Smoothness and wandering domains for torus maps, see  \cite{MR1375505}]
\label{thm:withnorton} 
   Let $f$ be a torus map of Denjoy type, and let $\Gamma \neq T^2$ be its minimal set. Then (i) if $f$ preserves a measurable, essentially bounded conformal structure on $\Gamma$, then the maps $f^n$, restricted to the prime end boundaries of the wandering domains, cannot be uniformly quasiconformal; and (ii) if $f$ preserves a $C^{2}$ conformal structure on $\Gamma$, then $f$ cannot be $C^{3}$. 
   %Here $C^{k+Z}$ denotes maps whose k-th derivative belongs to the Zygmund class. 
\end{theorem} 

One consequence of this theorem is that one cannot have a $C^3$  Denjoy diffeomorphism so that the iterates of some disc are all disjoint.  Of course this theorem is a partial higher dimensional analogue of Denjoy's famous theorem  showing that a $C^2$ diffeomorphism of the circle without periodic orbits cannot have wandering intervals, 
and therefore must be topologically conjugate to an irrational rotation. Again quite a few papers 
followed on from this work, for example: 
\begin{itemize}
\item it is not possible to have $C^1$ toral diffeomorphism with wandering round discs, see 
\cite{MR4015145};
\item very recently it was shown that there exist smooth and even real analytic 
diffeomorphisms of Denjoy type on the torus  with a wandering topological disc, see  
\cite{MR4015145,MR4345825};  see also \cite{MR4094968}.
\item wandering topological discs were also established for smooth two dimensional diffeomorphisms, 
see \cite{MR3685668},  and even for polynomial maps in higher dimensions \cite{MR3505180, MR3831030}. 
\end{itemize} 

Another theorem Dennis proved, jointly with Gambaudo and Tresser,  is the following (informally stated).

\begin{theorem}[Smoothness and linking number of periodic orbits for diffeomorphism of the disc, see  \cite{MR1259516}]
\label{thm:withgambaudo} Let $f$ be a $C^1$ diffeomorphism of the disc with periodic orbits $p_n$, $n\ge 0$, 
 so that for each $n\ge 0$ the periodic orbit $p_{n+1}$ {\lq}cycles as  a satelite{\rq} around $p_n$.
Then the average {\em linking number} between $p_{n+1}$ and $p_n$ must converge as $n\to \infty$. 
\end{theorem} 

One of the inspirations for this paper was a question by Smale, who asked whether it was possible to 
construct a smooth diffeomorphism on the disc with infinitely many hyperbolic periodic saddles, but without periodic sinks or sources (or neutral points). In \cite{MR431282,MR654885} such examples were constructed in the 
$C^1$ respectively $C^2$ category. It was subsequently shown that one can construct smooth and even real analytic diffeomorphisms with these properties, see \cite{MR994094} and also \cite{MR2192529,MR2836053}, namely with a Feigenbaum-Coullet-Tresser Cantor attractor. The constructions in those papers 
 build on the renormalisation theory developed for interval maps. The braid type of the periodic orbits
 is quite different from those in  \cite{MR431282,MR654885}, and as far as we know it is not yet known 
whether there are smooth diffeomorphisms which are topologically conjugate to the ones constructed in those
papers (in which it is conjectured that one cannot construct $C^3$ diffeomorphisms topologically conjugate 
to their examples). 

%In the proof of this result, the authors use some extension theorem
%about quasisymmetric maps $f\colon \R\to \R$ of the real line, and poses the question whether 
%there exists 
%%
%%
%%
%%   
%%   The method of proof includes a most interesting technique for extending a quasisymmetric map h of the real line to a quasiconformal map $E(h,G)$ of the upper half-plane, depending on the choice of a grid G on the line, and with the property that $E(f,g(G))\circ E(g,G)=E(f\circ g,G)$. 
%%   
%%   This extension places new interest on the following open question, which has been well known for a long time but is again emphasized by this paper: Can one find 
%   an extension $E(f)$ of $f$ to the complex plane such that $E(f\circ g)=E(f)\circ E(g)$ and such that if $f$ is $K$-quasisymmetric, then $E(f)$ is $K_1$-quasiconformal, where $K_1$ depends on $K$ only. 
%   Later on,  David Epstein and Markovic showed that this is not possible, see \cite{MR2302497}. CHECK QQQQ
%

What the theorems discussed in this section 
have in common is that they are about invariant structures, and that the full theory in this direction has not yet been completed. For this reason having the additional conformal structure was quite appealing to Dennis. 
Another reason to start working on holomorphic dynamics in the 1980's was of course that he saw a compelling analogue with the theory on Kleinian groups that he had been working on previously -- see section \ref{sec:kleinian} below. 

\subsection{Further results}

Another research direction we would like to explicitly mention is Dennis' work on currents, see \cite{MR433464, MR415679}. In his beautiful survey talk at the Abel Prize lectures at the University of Oslo  on Dennis' work, \'{E}tienne Ghys singles out this work (this talk is available on YouTube{\footnote{See \tt{https://www.youtube.com/watch?v=reC5-XUeH{\_}4}.}}).

\section{Dynamics and ergodic theory of Kleinian groups}
\label{sec:kleinian} \label{sec:kleiniangroups}

Dennis's interest in the geometric, dynamical and ergodic properties of discrete groups of hyperbolic isometries dates back to the mid to late 1970s. His work in this area was motivated in part by Mostow's rigidity theorem from two decades earlier, and in part by the work of Ahlfors on Kleinian groups, especially his famous finiteness theorem from 1965. 
Dennis was also greatly influenced by Thurston's work on geometric structures over $3$-manifolds. 
It was Lipman Bers who first told Dennis about the so-called \emph{Ahlfors conjecture}, according to which the limit set of a Kleinian group acting in hyperbolic $3$-space  either has zero Lebesgue measure in the sphere at infinity, or else it is equal to the entire sphere. This is now a theorem, thanks to the work of several mathematicians -- see for instance \cite{MR2355387} and references therein. 

Let us present a brief account of the contributions of Dennis to this beautiful subject.  Before we do that, we need to recall a few preliminary notions. For general background on the geometry of discrete groups and hyperbolic geometry, especially in dimension $3$, we recommend \cite{MR698777}, \cite{MR959135} and \cite{MR2355387}. For a systematic exposition of the work of Dennis (and Patterson) on the ergodic theory of discrete groups, see \cite{MR1041575}.

\subsubsection*{The Moebius group} 
Consider the one-point compactification $\widehat{\mathbb{R}}^n= \mathbb{R}^n\cup \{\infty\}$ of Euclidean $n$-space. 
The Moebius group in dimension $n$ is the group $MG(\mathbb{R}^n)$ consisting of all transformations $T: \widehat{\mathbb{R}}^n\to \widehat{\mathbb{R}}^n$ which arise as all possible compositions of (linear) conformal transformations of the form $x\mapsto Ax+b$, where $A$ is a scalar multiple of an orthogonal matrix and $b\in \mathbb{R}^n$, with the inversion $J: \widehat{\mathbb{R}}^n\to \widehat{\mathbb{R}}^n$ given by $J(x)=x/|x|^2$ for $x\neq 0$, $J(0)=\infty$ and $J(\infty)=0$. The elements of $MG(\mathbb{R}^n)$ are called \emph{Moebius transformations}. 

\subsubsection*{Hyperbolic space} Let us denote by $\mathbb{H}^n$ hyperbolic $n$-space, which we view as the open unit ball $B^n=\{x:\,|x|<1\}\subset  \mathbb{R}^n$ endowed with the \emph{hyperbolic metric} (also called \emph{Poincaré metric}) given by
\[
ds = \frac{2|dx|}{1-|x|^2 }\ .
\] 
The \emph{ideal boundary} or \emph{sphere at infinity} of hyperbolic $n$-space is by definition the sphere $S^{n-1}=\partial B^n$, endowed with the standard conformal structure inherited from $\mathbb{R}^n$. It is customary to denote the sphere at infinity by $S_\infty$.
The group $\isom^+(\mathbb{H}^n)$  of orientation-preserving isometries of this metric consists precisely of all \emph{Moebius transformations} that preserve the unit ball, \emph{i.e.,} those $T\in MG(\mathbb{R}^n)$ such that $T(B^n)=B^n$. Every  $T\in \isom^+(\mathbb{H}^n)$ acts on the sphere at infinity as a \emph{conformal automorphism}. 
The elements of $\isom^+(\mathbb{H}^n)$ are classified according to their action on $S_\infty$ as follows. If $T\in \isom^+(\mathbb{H}^n)$ has exactly one fixed point in $S_\infty$, then $T$ is said to be a \emph{parabolic} transformation. If it has exactly two fixed points in $S_\infty$, then it is called a \emph{loxodromic} transformation. All other elements of $\isom^+(\mathbb{H}^n)$ are said to be \emph{elliptic}. 

\subsubsection*{Kleinian groups}  A (generalized) \emph{Kleinian group} is a discrete subgroup $\Gamma\subset \isom^+(\mathbb{H}^n)$. Discreteness means in particular that the orbit $\Gamma(x)=\{\gamma x: \gamma\in \Gamma\}$ of any point $x\in B^n$ can only accumulate on the sphere at infinity. The set $\Lambda(\Gamma)\subseteq S_\infty$ of all such accumulation points is the \emph{limit set} of $\Gamma$ (see figure \ref{figure2}). Its complement $\Omega(\Gamma)=S_\infty \setminus \Lambda(\Gamma)$ is the \emph{domain of discontinuity} or \emph{ordinary set} of $\Gamma$. Clearly, both $\Lambda(\Gamma)$ and $\Omega(\Gamma)$ are completely invariant under the action of $\Gamma$. 
When $n=3$, the ordinary set $\Omega(\Gamma)$ is precisely the \emph{domain of normality} of $\Gamma$, that is to say, the set of all points $z\in S_\infty\equiv \widehat{\mathbb{C}}$ having a neighborhood $V_z\subset \widehat{\mathbb{C}}$ such that $\{\gamma|_{V_z} :\;\gamma\in \Gamma\}$ is a normal family in the sense of Montel (thus, $\Omega(\Gamma)$ is the analogue of the Fatou set for a rational map, and the limit set $\Lambda(\Gamma)$ is the analogue of the Julia set -- see \S {\ref{sullivandict}}).  A Kleinian group is said to be \emph{non-elementary} if its limit set consists of more than two points.

%%%%%%%%%%%%%%%%%%%%%%%%%%%%%%%%%%%%%%%%%

\begin{figure}[ht]
\begin{center}
\hbox to \hsize{
\includegraphics[width=4.2in]{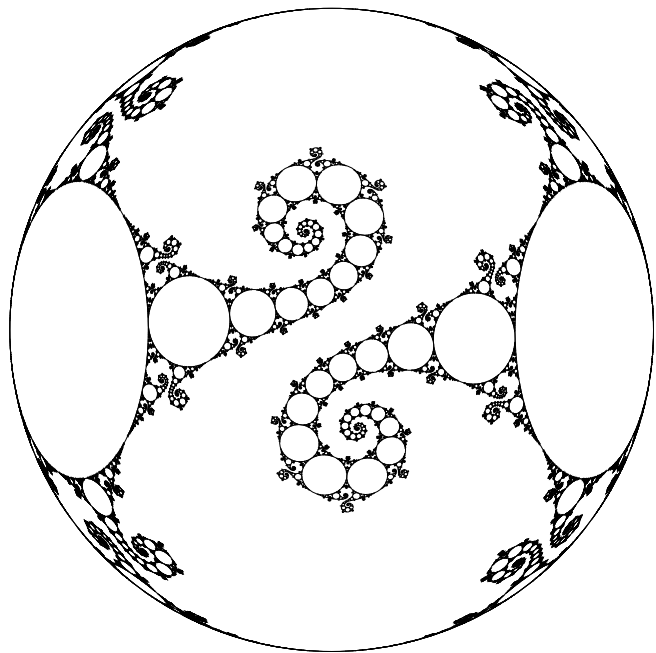}
}
\end{center}
\caption[figure2]{\label{figure2} The limit set of a Kleinian group sitting in the sphere at infinity is oftentimes a fractal object. [Credit: This picture was generated using C. McMullen's program "lim", available at https://people.math.harvard.edu/~ctm/programs/.]}
\end{figure}

%%%%%%%%%%%%%%%%%%%%%%%%%%%%%%%%%%%%%%%%

There are various ways under which the $\Gamma$-orbit of a point $x\in B^n$ can accumulate on a point of $\Lambda(\Gamma)$ in the sphere at infinity. The two most important are a \emph{conical} approach and a \emph{horospherical} approcah. Let us be more precise.

\begin{definition}
Let $\Gamma \subset \isom^+(\mathbb{H}^n)$ be a Kleinian group and let $\xi\in \Lambda(\Gamma)$.
\begin{enumerate}
\item[(i)] We say that $\xi$ is a \emph{conical limit point} of $\Gamma$ if for each $x\in B^n$ there exists a sequence $\{\gamma_n\}\subset \Gamma$ such that the ratio 
\[
\frac{|\xi-\gamma_n(x)|}{1-|\gamma_n(x)|}
\]
remains bounded as $n\to\infty$.
\item[(ii)] We say that $\xi$ is a \emph{horospherical limit point} of $\Gamma$ if for each $x\in B^n$ there exists a sequence $\{\gamma_n\}\subset \Gamma$ such that the ratio 
\[
\frac{|\xi-\gamma_n(x)|^2}{1-|\gamma_n(x)|}
\]
goes to zero as $n\to\infty$. 
\end{enumerate}
\end{definition}

It is an exercise to show that if $\xi$ is the fixed point of a loxodromic element of $\Gamma$, then $\xi$ is a conical limit point of $\Gamma$. 
The set of all conical limit points of $\Gamma$ is called the \emph{conical limit set}, and it is denoted $\Lambda_c(\Gamma)$. The set of all horospherical limit points of $\Gamma$ is called the \emph{horospherical limit set}, and it is denoted $\Lambda_h(\Gamma)$. Clearly, these are  both $\Gamma$-invariant. 

An important class of Kleinian groups is the class consisting of so-called \emph{convex co-compact groups}. Given a (non-elementary) Kleinian group $\Gamma$, let $\Lambda=\Lambda(\Gamma)$ be its limit set, and consider the convex hull $C(\Lambda)$ of $\Lambda$ inside hyperbolic space. Then $C(\Lambda)$ is invariant under $\Gamma$, and we say that $\Gamma$ is convex co-compact if the quotient $C(\Lambda)/\Gamma$ is compact. 
It is not difficult to see that if $\Gamma$ is convex co-compact, then every element of $\Lambda$ is a conical limit point -- in other words, $\Lambda_c(\Gamma)=\Lambda(\Gamma)$ in this case.

\subsubsection*{Hyperbolic manifolds} The quotient space $M_\Gamma = \mathbb{H}^n/\Gamma$ of hyperbolic $n$-space by a Kleinian group $\Gamma$ is what one calls an \emph{orbifold} (after Thurston). Such quotient is always a manifold when $n=2,3$, but it may fail to be one when $n>3$. However, if $\Gamma$ acts \emph{freely and properly discontinuously} on $\mathbb{H}^n$, then $M_\Gamma$ is indeed a manifold. Such manifolds are called \emph{hyperbolic}. 
The natural quotient projection $\mathbb{H}^n\to M_\Gamma$ is a proper covering map, and therefore the hyperbolic metric of $\mathbb{H}^n$ descends to $M_\Gamma$. Thus, $\mathbb{H}^n$ is the universal covering space of $M_\Gamma$, and the fundamental group $\pi_1(M_\Gamma)$ is (isomorphic to) $\Gamma$. 
It is not difficult to see that if two Kleinian groups $\Gamma_1, \Gamma_2$ are \emph{conjugate} subgroups of $\isom^+(\mathbb{H}^n)$, \emph{i.e.,} if there exists $\gamma\in \isom^+(\mathbb{H}^n)$ such that $\Gamma_1=\gamma^{-1}\Gamma_2\gamma$, then the corresponding orbifolds $M_{\Gamma_i}$, $i=1,2$, are \emph{isometric}, and conversely. 

\subsubsection*{Quasi-conformal homeomorphisms}
A \emph{quasiconformal homeomorphism} $h: S_\infty\to S_\infty$ is a homeomorphism which is differentiable Lebesgue almost-everywhere, and whose derivative at each point of differentiability  maps round spheres onto ellipsoids whose ratios between the largest axis and the smallest axis yield a measurable function on the sphere that is essentially bounded. The essential norm of this function is called the quasi-conformal distortion of $h$, denoted $K_h$. 
It turns out that if $K_h=1$ then $h$ is in fact \emph{conformal}.  Every quasi-conformal homeomorphism determines a measurable field of ellipsoids, also known as a \emph{measurable conformal structure} on $S_\infty$. In dimension two, such measurable conformal structures can be integrated to recover $h$ up to post-composition by a conformal map -- a famous result known as the measurable Riemann mapping theorem -- but no such theorem exists in higher dimensions. 

\subsection{Sullivan's rigidity theorem} It is a truly remarkable theorem due to G.~Mostow \cite{MR236383} that complete finite-volume hyperbolic $n$-manifolds are determined up to isometry by their fundamental groups when $n\geq 3$. This is the famous \emph{Mostow rigidity theorem}, which can be formally stated as follows.

\begin{theorem}[Mostow Rigidity] 
Let $M$ and $N$ be two complete, finite-volume hyperbolic $n$-manifolds with $n\geq 3$, and let $\theta:\pi_1(M)\to \pi_1(N)$ be an isomorphism between their fundamental groups. Then there exists an isometry $f:M\to N$ between both manifolds such that the induced isomorphism $f_{*}:\pi_1(M)\to \pi_1(N)$ agrees with $\theta$. 
\end{theorem}

In fact, this theorem was proved by Mostow for \emph{closed} manifolds, \emph{i.e.,} compact manifolds without boundary. It was then extended to finite volume manifolds by Marden \cite{MR349992} in dimension $n=3$, and by Prasad \cite{MR385005} in all dimensions $n\geq 3$. 

The way Mostow proved his theorem was by first showing that the manifolds $M, N$ are \emph{pseudo-isometric} in the following sense. A continuous, surjective map $\phi: M\to N$ between two hyperbolic manifolds is a \emph{pseudo-isometry} if (a) it induces an isomorphism between the fundamental groups of both manifolds and moreover (b) there exist constants $K>1$ and $\delta>0$ such that 
\[
\frac{1}{K} \leq \frac{d_N(\phi(x),\phi(y))}{d_M(x,y)} \leq K
\]
for each pair of points $x,y\in M$ such that $d_M(x,y)\geq \delta$. 
Here $d_M,d_N$ denote the hyperbolic distances in $M$ and $N$, respectively.
Condition (b) is saying that a pseudo-isometry distorts hyperbolic distances between points by a bounded amount, provided these points are sufficiently far apart. 

If one lifts a given pseudo-isometry between $M$ and $N$ to their universal covering space, one gets a pseudo-isometry of hyperbolic $n$-space. Mostow proved in \cite{MR236383} that every pseudo-isometry of $\mathbb{H}^n$ extends continuously to the sphere at infinity as a \emph{quasiconformal homeomorphism} $h: S_\infty\to S_\infty$. 
The key step in the proof is to show by means of an ergodic argument that if $\Gamma$ is a finite-volume Kleinian group then there is only one measurable conformal structure on $S_\infty$ which is $\Gamma$-invariant. This implies that the quasiconformal homeomorphism $h$ is actually conformal, which in turn means that the two associated Kleinian groups $\Gamma_M\simeq \pi_1(M), \Gamma_N\simeq \pi_1(N)$ are conjugate subgroups of $\isom^+(\mathbb{H}^n)$, and therefore the hyperbolic manifolds $M$ and $N$ must be isometric. 

Here is another way of stating Mostow's theorem. Following \cite{MR624833}, we say that a hyperbolic manifold $M$ is \emph{Mostow-rigid} if any pseudo-isometry between $M$ and another hyperbolic manifold $N$ is homotopic to an isometry. Recast in this language, Mostow's theorem states that every complete, finite-volume hyperbolic $M$ is Mostow-rigid. 

Dennis proved in \cite{MR624833}, in his own words, a \emph{maximal extension} of Mostow's theorem. In order to state his theorem, let us introduce some notation. Given a hyperbolic manifold $M$, a point $p\in M$ and $r>0$, let $V_M(p,r)$ denote the hyperbolic volume of the set $\{x\in M:\,d_M(p,x)<r\}$. For example, for the hyperbolic ball of radius $r$ in hyperbolic $n$-space, we have
\[
V_{\mathbb{H}^n}(0,r) = \omega_n \int_{0}^{\tanh{(r/2)}} \frac{2^n|x|^{n-1} d|x|}{(1-|x|^2)^n} = 
\omega_n\int_{0}^{r} \sinh^{n-1}{t}\,dt \sim  \mathrm{const}\cdot e^{(n-1)r}\ .
\]
Here, $\omega_n$ denotes the euclidean area of $S_\infty = S^{n-1}$. 

\begin{lemma}\label{undergrowth}
Let $M=\mathbb{H}^n/\Gamma$ be a hyperbolic $n$-manifold. Then the following assertions are equivalent.
\begin{enumerate}
\item[(i)] The ratio $V_M(p,r)/V_{\mathbb{H}^n}(0,r)$ goes to zero as $r\to\infty$. 
\item[(ii)] The Kleinian group $\Gamma$ acts conservatively on $S_\infty$.
\item[(iii)] The horospherical limit set of $\Gamma$ has full Lebesgue measure on $S_\infty$.
\end{enumerate}
\end{lemma}

A fourth assertion equivalent to the above three is that the fundamental domain of $\Gamma$ in $\mathbb{H}^n$ has zero Lebesgue measure on $S_\infty$. 
We now have all the necessary elements to state the Sullivan rigidity theorem. 

\begin{theorem}[Sullivan Rigidity] Let $M$ be a complete hyperbolic $n$-manifold, and suppose that $M$ satisfies any of the equivalent conditions of Lemma \ref{undergrowth}. Then $M$ is Mostow rigid.
\end{theorem}

Dennis deduces this theorem from the following result, also due to him.

\begin{theorem}\label{noinvariantlinefields}
Let $\Gamma\subset \isom^+(\mathbb{H}^n)$ be a Kleinian group, and consider its action on the sphere at infinity. Suppose $\nu$ is a measurable conformal structure (\emph{i.e.,} a measurable field of ellipsoids) which is Lebesgue almost everywhere invariant under $\Gamma$. Then $\nu$ agrees a.e. with the standard conformal structure of $S_\infty$ on the limit set $\Lambda_\Gamma$. 
\end{theorem}

In the case of hyperbolic $3$-manifolds, \emph{i.e.,} when $n=3$, a major consequence of this theorem is obtained by combining it with the Ahlfors finiteness theorem. Recall that a Riemann surface $X$ is said to be of \emph{finite type} if $X$ is obtained from a compact Riemann surface by removing from it a finite set of points. 

\begin{theorem}[Ahlfors Finiteness Theorem]
Let $\Gamma\subset PSL(2,\mathbb{C})$ be a finitely generated Kleinian group. Then $\Omega(\Gamma)/\Gamma$ is a finite union of Riemann surfaces of finite type. 
\end{theorem}

In particular, the \emph{Teichm{\"u}ller space} $\mathrm{Teich}(\Omega(\Gamma)/\Gamma)$ is finite-dimensional. 
Hence we have the following result.

\begin{corollary}
Let $\Gamma\subset PSL(2,\mathbb{C})$ be a finitely generated Kleinian group. Then the space of quasi-conformal deformations of $\Gamma$ is parametrized by $\mathrm{Teich}(\Omega(\Gamma)/\Gamma)$, and is therefore finite-dimensional. 
\end{corollary}

In particular, if $\Gamma$ has a dense orbit on the sphere $S^2$, then the space of quasi-conformal deformations of $\Gamma$ reduces to a point, \emph{i.e.,} $\Gamma$ is quasi-conformally rigid. 

\subsection{Conformal densities and Patterson-Sullivan measures} 

Let $\Gamma$ be a non-elementary Kleinian group acting on $\mathbb{H}^n\cup S_\infty$, and as before let $\Lambda(\Gamma)\subseteq S_\infty$ be its limit set. Also as before, let $\Lambda_c(\Gamma)\subseteq \Lambda(\Gamma)$ be its conical limit set.
Generalizing work of Patterson for Fuchsian groups \cite{MR450547}, Dennis was able to construct, in \cite{MR556586}, an invariant measure for the geodesic flow on the unit tangent bundle of the hyperbolic manifold $\mathbb{H}^n/\Gamma$. This measure comes from a \emph{conformal density} $\mu$ on the sphere at infinity, and the geodesic flow is either ergodic or dissipative, depending on whether $\mu$ assigns positive or zero measure to $\Lambda_c(\Gamma)$, respectively. 
We proceed to a brief description of the construction. Details can be found either in the original paper by Dennis, or in the book by Nicholls \cite{MR1041575}.

\subsubsection*{Conformal densities}
Let us start by clarifying what is meant by conformal density. 
Let $M$ be a smooth manifold, let $\mathcal{R}$ be a non-empty collection of Riemannian metrics on $M$, and let $\alpha>0$. 
Following \cite[p.~421]{MR556586}, we define a \emph{conformal density of dimension $\alpha$}, or \emph{$\alpha$-conformal density}, on $M$ (relative to $\mathcal{R}$) to be a function that assigns to each element $g\in \mathcal{R}$ a positive, finite Borel measure $\mu_g$ on $M$ in such a way that, whenever $g_1$ and $g_2$ are in the same conformal class (\emph{i.e.,}, whenever $g_1=\varphi g_2$ for some positive function $\varphi$), then $\mu_{g_1}$ and $\mu_{g_2}$ are in the same measure class, and the Radon-Nikodym derivative $d\mu_{g_1}/d\mu_{g_2}$ satisfies
\[
\frac{d\mu_{g_1}}{d\mu_{g_2}} = \left( \frac{g_1}{g_2}\right)^{\alpha}\ .
\]
We are not interested in conformal densities in such vast generality, but rather in the following specific context. We take $M=S_\infty$, and let $\mathcal{R}=\{g_x:\,x\in B^n\}$, where $g_0$ is the standard (euclidian) Riemannian metric on the sphere $S_\infty$, and for each $x\in B^n$ the Riemannian metric $g_x$ is obtained by transporting $g_0$ via any hyperbolic isometry mapping $0$ to $x$.
These metrics are all conformally equivalent.  
Thus, in the present context, we can think of a $\alpha$-conformal density on the sphere at infinity as an assigment $x\mapsto \mu_x$ from points on hyperbolic space to positive measures on $S_\infty$, all in the same measure class.
Shortening the notation to  $\mu_x=\mu_{g_x}$, we deduce after a simple calculation that 
\begin{equation}\label{radon-nikodym}
\frac{d\mu_{x_1}}{d\mu_{x_2}}(\xi) =\left(\frac{P(x_1,\xi)}{P(x_2,\xi)} \right)^{\alpha}\ ,
\end{equation}
for each pair of points $x_1,x_2\in B^n$ and all $\xi\in S^{n-1}\equiv S_\infty$, where 
\[
P(x,\xi) = \frac{1-|x|^2}{|x-\xi|^2}
\]
is the well-known \emph{Poisson kernel}. 

Given a non-elementary Kleinian group $\Gamma$, we are interested in conformal densities of the type just described  \emph{that are entirely supported in the limit set of $\Gamma$ and that are $\Gamma$-invariant}. More precisely, we want to know whether there exists an $\alpha$-conformal density $\mu=\{\mu_x:\,x\in B^n\}$ such that
\begin{enumerate}
\item[(1)] For each $x$, the measure $\mu_x$ has support in the limit set $\Lambda(\Gamma)$.
\item[(2)] For each pair of points $x_1,x_2\in B^n$, the measures $\mu_{x_1},\mu_{x_2}$ are mutually absolutely continuous, and the Radon-Nikodym derivative $d\mu_{x_1}/d\mu_{x_2}$ satisfies \eqref{radon-nikodym}.
\item[(3)] For all $x\in B^n$ and each $\gamma\in \Gamma$, we have $\gamma_*\mu_{x}=\mu_{\gamma x}$.  
\end{enumerate}

Given a $\Gamma$-invariant conformal density of dimension $\alpha$ in the sense just described, each of its associated measures $\mu_x$ is an \emph{$\alpha$-conformal measure} in the sense that 
\[
\mu_x(\gamma (E)) = \int_{E} |\gamma_x^{\,\prime} (\xi)|^{\alpha}\,d\mu_x(\xi)
\]
for each Borel set $E\subset S_\infty$ and each $\gamma\in \Gamma$ (\emph{cf.} the discussion on conformal measures for rational maps in \S\ref{sec:conf-meas-rat}).
The question as to whether such \emph{Patterson-Sullivan measures} exist is examined below. 

\subsubsection*{Patterson-Sullivan measures: construction}
 The \emph{Poincaré series} of the (non-elementary) Kleinian group $\Gamma$ is defined as 
\begin{equation}\label{poincareseries}
g_s(x,y) = \sum_{\gamma\in \Gamma} e^{-sd(x,\gamma y)}\ ,
\end{equation} 
where $x,y\in \mathbb{H}^n$, $d$ is the hyperbolic metric on $\mathbb{H}^n$, and $s>0$ is a real parameter. Whether the series \eqref{poincareseries} converges or not for a given value of $s$ is independent of which points $x,y$ one chooses. In order to state this more precisely, define the \emph{critical exponent}  of $\Gamma$ to be the number
\[
\delta(\Gamma) = \inf\{s>0:\,g_s(0,0)<\infty\}\ ,
\]
which turns out to be strictly positive when $\Gamma$ is non-elementary{\footnote{This was first proved by Beardon \cite{MR227402}.}}. Then it is a fact that the series \eqref{poincareseries} converges for all $s>\delta(\Gamma)$, and diverges for all $0<s<\delta(\Gamma)$. It is also not difficult to prove that $\delta(\Gamma)\leq n-1$. 
What happens when $s=\delta(\Gamma)$? Obviously, only one of two things: 
\begin{enumerate}
\item[(i)] if $g_{\delta(\Gamma)}(0,0)<\infty$, we say that $\Gamma$ is a group of \emph{convergence type}; 
\item[(ii)] if $g_{\delta(\Gamma)}(0,0)=\infty$, we say that $\Gamma$ is a group of \emph{divergence type}. 
\end{enumerate}

For the construction to follow, let us fix a point $y\in B^n$ once and for all (for example, we could take $y=0$). For each $x\in B^n$ and each $s> \delta(\Gamma)$, one constructs a positive Borel measure $\mu_{x,s}$ on the closure of $B^n$ as follows. The rough idea is to place a Dirac mass at each point of the $\Gamma$-orbit of $y$, with weights that depend on the hyperbolic distance between each such point and $x$ in a suitable way. Let us be more precise.

When the group $\Gamma$ is of \emph{divergence} type, one simply defines{\footnote{We denote by $\delta_z$ the Dirac probability measure concentrated at $z\in B^n$.}}
\[
\mu_{x,s} = \frac{1}{g_s(y,y)} \sum_{\gamma\in \Gamma} e^{-sd(x,\gamma y)}\delta_{\gamma x}
\]
With this definition, one can consider the weak limits of such measures when 
$s\searrow \delta(\Gamma)$. Since $g_s(y,y)\to \infty$ as $s\to \delta(\Gamma)$, the point masses are swept off to the sphere at infinity, and any weak limit will be a measure supported on the sphere (actually on the limit set). The existence of limits is guaranteed by a classical result in real analysis (namely, Helly's theorem). 

However, when the group is of \emph{convergence} type, the above will not work, because 
$g_s(y,y)$ remains bounded as $s\to \delta(\Gamma)$, and whatever limiting measure we get will still have an atom at each point in the $\Gamma$-orbit of $y$ (recall that the goal is to obtain measures supported on the limit set of $\Gamma$). To circumvent this problem, Dennis borrows an idea due to Patterson \cite{MR450547} (in the Fuchsian case, $n=2$) and introduces a \emph{mollifier}, called, not surprisingly, the \emph{Patterson auxiliary function}. This is a continuous non-decreasing function $h: \mathbb{R}^+\to  \mathbb{R}^+$ having the following properties:
\begin{enumerate}
\item[(1)] For each $\epsilon>0$ there exists $r_0>0$ such that $h(tr)\leq t^\epsilon h(r)$ for all $r>r_0$ and all $t>1$.
\item[(2)] The series 
\[
\sum_{\gamma\in \Gamma} e^{-s d(x, \gamma y)} h( e^{d(x,\gamma y)})
\]
converges for $s> \delta(\Gamma)$ and diverges for $s\leq \delta(\Gamma)$. 
\end{enumerate}
Using this function, one defines the \emph{modified Poincaré series}
\[
g_s^*(x,y) = \sum_{\gamma\in \Gamma} e^{-sd(x,\gamma y)} h( e^{d(x,\gamma y)})
\]
This now diverges when $s\searrow \delta(\Gamma)$. Thus, for each $x\in B^n$ and each $s> \delta(\Gamma)$ we may now consider the positive Borel measure 
$\mu_{x,s}$ defined by
\[
\mu_{x,s} = \frac{1}{g_s^*(y,y)} \sum_{\gamma\in \Gamma} e^{-sd(x,\gamma y)}h( e^{d(x,\gamma y)})\delta_{\gamma x}\ .
\]
As before, the resulting weak limits as $s\searrow \delta(\Gamma)$ are positive Borel measures on the sphere at infinity, and their supports are contained in the limit set $\Lambda(\Gamma)$. 

In either case, we have for each $x\in B^n$ a non-empty closed subset $M_x(\Gamma)\subset \mathcal{M}^+(S_\infty)$ of the space of all positive Borel measures on the sphere at infinity (endowed with the topology of weak convergence of measures). 
It is possible to prove that the $M_x(\Gamma)$'s are all homeomorphic (see for instance \cite[Th.~3.4.1]{MR1041575}). 

We are now ready to summarize some of the main results obtained by Dennis in \cite{MR556586}. That paper is very rich, and we can hardly do any justice to it in such a short exposition.

The first theorem generalizes results obtained by Patterson and Bowen in the Fuchsian case.

\begin{theorem}[Patterson-Sullivan Measures and Hausdorff Dimension, see \cite{MR556586}]
Let $\Gamma\subset \isom^+(\mathbb{H}^n)$ be a non-elementary  Kleinian group, and let $\delta=\delta(\Gamma)$ be its critical exponent. 
\begin{enumerate}
\item[(i)] There exists a $\delta$-conformal density $\mu=\{\mu_x:\,x\in B^n\}$ on the sphere at infinity which is $\Gamma$-invariant and satisfies $\mu_x\in M_x(\Gamma)$ for each $x$.
\item[(ii)] We have $\mathrm{dim}_H( \Lambda_c(\Gamma))\leq \delta$, \emph{i.e.,} the conical limit set of $\Gamma$ has Hausdorff dimension less than or equal to its critical exponent
\item[(iii)] If $\Gamma$ is convex co-compact, then the $\delta$-conformal density in (i) is unique up to a scalar multiple, and for each euclidian ball $B(\xi,r)$ centered at a point $\xi\in \Lambda_c(\Gamma)$, and each $x\in B^n$, we have 
$\mu_x(B(\xi,r)\cap \Lambda_c(\Gamma)) \asymp r^{\delta}$. In particular, 
$\mathrm{dim}_H( \Lambda_c(\Gamma)) = \delta$,\emph{i.e.,} the Hausdorff dimension of the conical limit set is \emph{equal} to $\delta$ in this case.
\end{enumerate}
\end{theorem}

The last item in the above statement is a very elegant result which generalizes an equally elegant result for Fuchsian groups due to Bowen \cite{MR556580}.

%%%%%%%%%%%%%%%%%%%%%%

Another striking result obtained by Dennis in \cite{MR556586} states that the \emph{total-mass function} of a $\Gamma$-invariant conformal density in dimension $\delta(\Gamma)$ is an eigenfunction of the \emph{hyperbolic Laplacian}. Let us state this result a bit more precisely, explaining the meaning of these terms. The hyperbolic Laplacian in $\mathbb{H}^n\equiv B^n$ (written in generalized polar coordinates) is the second-order differential operator 
\[
\Delta_h\;=\; \frac{(1-r^2)^2}{4}\left[\Delta + \frac{2(n-2)r}{1-r^2} \frac{\partial}{\partial r}  \right]\ ,
\]
where $\Delta$ is the standard (euclidian) Laplacian -- see for instance \cite[p.~56]{MR725161}. 
If $\mu=\{\mu_x:\,x\in B^n\}$ is a $\Gamma$-invariant conformal density in dimension $\delta(\Gamma)$, its \emph{total-mass function} is the function $\varphi: \mathbb{H}^n\to \mathbb{R}$ given by 
\[
\varphi(x) = \int_{\partial B^n} d\mu_x(\xi) = \int_{S_\infty} d\mu_x(\xi)\ .
\]

\begin{theorem}
The total-mass function $\varphi$ is an eigenfunction of the hyperbolic Laplacian, \emph{i.e.,} it satisfies
\[
\Delta_h\varphi = \lambda(\Gamma)\varphi \ ,
\]
with eigenvalue $\lambda(\Gamma)=\delta(\Gamma)\left( 1+ \delta(\Gamma) -n  \right)$.
\end{theorem}

This statement is perhaps made more plausible if one takes into account  that, for each $\alpha$, we have
\[
\Delta_h\left( P(x,\xi)^\alpha \right) = \alpha(\alpha -n+1) P(x,\xi)^\alpha
\]
a fact that can easily be checked by direct calculation.

Finally, as Dennis explains in \cite{MR556586}, a $\Gamma$-invariant conformal density $\mu$ gives rise to a measure $m$ on the unit tangent bundle of the quotient hyperbolic manifold $\mathbb{H}^n/\Gamma$ which is \emph{invariant under the geodesic flow}. In addition, the normalized probability measures $\varphi(x)^{-1}\mu_x$ can be used to generate a Markovian stochastic process on $\mathbb{H}^n/\Gamma$ akin to Brownian motion. The beautiful synthesis obtained by Dennis in \cite{MR556586} relates the recurrent properties of this Markovian process with the ergodic properties of the geodesic flow on the quotient manifold, and can be informally stated as follows.

\begin{theorem}[Ergodic Measures for the Geodesic Flow, see \cite{MR556586}]
Let $\Gamma\subset \isom^+(\mathbb{H}^n)$ be a non-elementary  Kleinian group, and let $\delta=\delta(\Gamma)$ be its critical exponent. Also, let $\mu$ be a $\Gamma$-invariant $\delta$-conformal density. 
Consider the following assertions:
\begin{enumerate}
\item[(1)] The conical limit set $\Lambda_c(\Gamma)$ has positive $\mu$-measure.
\item[(2)] The action of $\Gamma$ on $S_\infty \times S_\infty$ minus the diagonal is ergodic with respect to $\mu\times \mu$.
\item[(3)] The geodesic flow on the unit tangent bundle to $\mathbb{H}^n/\Gamma$ is ergodic with respect to $m_\mu$.
\item[(4)] The group $\Gamma$ is of divergence type.
\item[(5)] The Markov process on $\mathbb{H}^n/\Gamma$ is recurrent.
\end{enumerate}
Then (1), (2) and (3) are equivalent, and they imply (4). If in addition $2\delta>n-1$, then (4) implies (5), and (5) implies all the others (\emph{i.e.,} all five assertions are equivalent in this case). 
\end{theorem}

\subsection{Further results}

Dennis has written a number of other very interesting papers on the geometry and dynamics of Kleinian groups. For example, in \cite{MR766265} he extended some of the above results from the convex co-compact case to the case of \emph{geometrically finite groups}  -- one of the main consequences being the fact that the Hausdorff dimension of the limit set of a geometrically finite group is equal to the critical exponent of the group. In \cite{MR688349}, he investigated the excursion of geodesics on hyperbolic surfaces, relating their behaviour near cusps with classical results in Diophantine approximations.

\section{Holomorphic dynamics}

\subsection{The reemergence of holomorphic dynamics in Paris in the 1980's} 
Right from the start the city of Paris has played a key role in
the study of holomorphic dynamical systems. In the 1920's Julia and Fatou 
developed many results on the theory of iterations of rational maps  $f\colon \bar \C \to \bar \C$.
They introduced what is now called the {\em Fatou set}, 
the set of points in $\bar \C$ which have a neighbourhood $N$ on which the iterates $f^n|N$, $n\in \mathbb N$, 
form a normal family, i.e. are equicontinuous. Similarly,  the Julia set $J(f)$ is defined as the complement of $F(f)$.
At the time the main tool Julia and Fatou had at their disposal was the  Montel theorem, which states 
that a family of maps $f^n|N$, $n\in \mathbb N$ defined on an open set $N\subset \bar \C$ 
is normal if it has the property  that there are three points in $\bar \C$ which are omitted in $\cup_n f^n(N)$. 
Among the many results they showed is that the Julia set is the closure of the set of
repelling periodic points. They also developed a theory on the local dynamics near 
periodic points. 

In the early 1980's there was a huge revival of this theory in Paris, with main 
drivers being  Dennis Sullivan, Adrien Douady, Hamal Hubbard and Michael Herman. 
One of the main reasons for this resurgence was that it became increasingly 
clear that there were {\em new powerful tools} available, namely  the {\em Measurable Riemann Mapping Theorem}  (MRMT)
and the notion of {\em quasiconformal maps}. These 
  would make it possible to complete and go much beyond the theory initiated by Julia and Fatou in the 1920's. 

There are quite a few equivalent definitions of the notion of a  {\em quasiconformal map}, 
and all reflect that such maps are generalisations  of conformal maps. 
One of these definitions is that  $h\colon U \to V$ is a quasiconformal map if it is an orientation preserving homeomorphism
between two domains $U,V$ on $\C$ so that the Beltrami equation 
\[
\dfrac{\partial h}{\partial {\bar z}} = \mu(z) \dfrac{\partial h}{\partial z}
\]
makes sense and so that $\mu$ is Lebesgue measurable essentially bounded, \emph{i.e.,} satisfies 
$\|\mu\|_\infty<1$. This means that at each point $z\in U$ at which $h$ is differentiable, the derivative $Dh(z)$ maps circles to ellipses with uniformly bounded eccentricity. 

The MRMT  implies that each such $\mu$ is associated to a quasiconformal map $h_\mu$ and, crucially, that $h_\mu$ depends analytically on $\mu$. 

%\begin{theorem} [MRMT] 
%\end{theorem} 

Dennis was amongst the first to realise the power of the MRMT in the field
of holomorphic dynamics, partly because he had previously used it very successfully
in the study of Kleinian groups. 
Parallel to Douady and Hubbard's seminal  Orsay notes \cite{MR651802,MR762431,MR812271, MR816367},   which are a tour the force
through the entire subject of holomorphic dynamics,  Dennis proved the following remarkable theorem:

\begin{theorem}
[No-wandering-domains Theorem, see \cite{MR819553}] \label{thm:nowandering} 
Let $f$ be a rational map on the Riemann sphere.  Then each component of the Fatou set is 
eventually periodic. 
\end{theorem}

More precisely, for each component  $U$ of $F(f)$ there exists $m\ge 0$ so that $V=f^m(U)$ is periodic, 
i.e. there exists $p\ge 1$ so that $f^p(V)=V$.  Moreover, each periodic component $V$ of $F(f)$ can be classified. 
Indeed, it  contains one of the following:
\begin{enumerate}
\item  a periodic point of eigenvalue $\lambda=0$ (called {\em superattractive}), 
\item a periodic point of eigenvalue $0<|\lambda|<1$ (called {\em attractive}), or
\item  $\partial V$ contains a periodic point whose multiplier is a root of unity ({\em rational indifferent}), 
\item $f^p$ is analytically conjugate to an irrational rotation on $V$
and either 
\begin{enumerate} 
\item \quad  $V$ is a simply-connected {\em Siegel disc} or 
\item \quad a doubly-connected {\em Herman ring}. 
\end{enumerate} 
\end{enumerate} 

To prove the first part of the statement one needs to 
show that if $f$ is a rational map, then no component of its Fatou set
is a wandering domain, i.e.  a domain so that all its iterates are pairwise disjoint.
In a nutshell Dennis' proof of this theorem goes as follows:
suppose by contradiction that $f$ has a wandering domain $W$.  
Then this makes it possible to construct
an infinite dimensional space of deformations of $f$, 
contradicting that the space of rational maps of a given degree
is finite dimensional.  

That the space of rational maps is finite dimensional is crucial: 
shortly after the preprint version of \cite{MR819553} appeared,  Baker constructed 
entire functions $f\colon \C\to \C$ which do have wandering domains. 
Sullivan's no wandering theorem has also been extended to the setting
of entire maps (and similar spaces) with a finite number of singular values, see
for example   \cite{MR1196102} and \cite{MR857196}.  A very elegant proof of the no wandering domains theorem due to McMullen -- which circumvents the use of the MRMT and uses an infinitesimal deformations argument more in line with Ahlfors' original proof of his finiteness theorem -- can be found in \cite[p.~90]{mcmullennotes}. 

Unfortunately, as there is no corresponding MRMT in the real one-dimensional case,  the analogous theory in the real one-dimensional case
requires a careful combinatorial analysis together with an understanding 
of the non-linearity of the map. For circle diffeomorphisms this goes back
to Denjoy in the 1930's and from this paper Dennis learned the {\em smallest
interval argument}:  Assume that $W$ is a maximal wandering interval, i.e. that $W$ 
is not contained in a larger wandering interval. 
Then for each $n\ge 3$ consider the smallest, say $f^i(W)$,   amongst the 
collection of disjoint intervals $W,\dots,f^n(W)$. Then $f^i(W)$
has neighbours on each side which are larger (or empty space in the case of an interval map).  So $f^i(W)$ 
is well-inside  the convex hull $W_i'$ of the two neighbours. Using the 
way $W_i'$ is chosen, the interval $W_i'$ can be pulled back to an interval $W'_0\supset W$ 
so that the pullbacks $W_0',\dots,W_i'$ are essentially disjoint. 
This disjointness and the fact that $f$ is a $C^2$ diffeomorphism
implies that $W$ is $\delta$-well-inside $W_0'$ where $\delta$ does not depend on $n$. 
Using the maximality of $W$ this gives a contradiction. That  $f^i(W)$ 
is well-inside  the convex hull $W_i'$ is often called {\em Koebe space} and is a property 
that is often used both in real and holomorphic dynamics. 

The terminology {\em Koebe space} comes from the Koebe Lemma in complex analysis, which states
that for each $\delta>0$ there exists $K>0$ so that when $U_0\subset U$ are topological discs  
and $\phi\colon U\to \C$ is a univalent map and the modulus of $U\setminus U_0$
is at least $\delta$ then 
$$\frac{|D\phi(z)|}{|D\phi(z')|}\le K\mbox{ for all }z,z'\in U_0.$$ 
The power of this Lemma is that $K$ does not depend on $\phi$. 
It turns out that  in the real case there is an analogous result: 
for each $\delta>0$ there exists $K>0$ so that when $I_0\subset I$ are intervals
and  $g\colon I\to \R$ is a diffeomorphism so that  $Sg\ge 0$
then 
$$\frac{|Dg(z)|}{|Dg(z')|}\le K\mbox{ for all }z,z'\in I_0.$$
In applications, $g$ is usually the inverse of a diffeomorphic branch of an interval map $f^n$
where $f$ is assumed to have negative Schwarzian $Sf<0$.  The reason this is useful is that 
the Schwarzian property has the property that $Sf<0$ implies $Sf^n<0$
and that $Sf^n<0$ implies $Sf^{-n}>0$ (on diffeomorphic branches). Furthermore, 
if $f$ is a real polynomial with only real critical points then $Sf<0$. 
This observation that the negative Schwarzian could 
be used to bound the number of periodic attractors of an interval map 
was first made by Singer, see \cite{MR494306}
but also appeared at around the same time in for example Herman's work \cite{MR538680}. 
A version of the Koebe Lemma in this setting was proved for the first time in \cite{MR654898}. 
That the Schwarzian derivative is related the distortion of cross-ratio
was already known by E. Cartan in the 1930's, see the discussion in \cite[Sections IV.1 and IV.2]{MR1239171}. 
As the use of the above distortion estimate is so widespread, one often refers to the {\em Koebe Principle}
and the assumption on the domains $U_0\subset U$ (resp. $I_0\subset I$) as {\em Koebe space}. 

 For interval maps and critical circle maps, the presence
of critical points implies that one cannot control the non-linearity of the map 
and its iterates. Instead,  it turns out that it is enough (i) to consider the 
{\em cross-ratio distortion} of a triple of adjacent intervals under iterates, (ii) assume that  
the map has some local symmetry around the critical points (e.g. the maps
are non-flat at the critical points) and (iii) a more elaborate combinatorial analysis
of orbits of wandering intervals.  
This was done by Guckenheimer, Yoccoz, Lyubich, Block, de Melo, van Strien, Martens in various generalities. For a history and a full analogue
of  Theorem~\ref{thm:nowandering}  see  \cite{MR1161268, MR1239171}. 
Probably the most elegant way of proving 
 absence of wandering intervals in this setting can be found in \cite{MR2083467}. 
 Interestingly, it was Dennis who emphasised and insisted that the right smoothness class
 for (i) is $C^{1+Zygmund}$, whereas the earlier results required that the map 
 was $C^3$ and even assumed that the  map has negative Schwarzian. See \cite{MR1209848}, \cite{MR1184622} or \cite{MR1239171}
 for the definition of   the classes $C^{1+Zygmund}$ and $C^{1+zygmund}$.  The analogue
 of Sullivan's no-wandering Theorem~\ref{thm:nowandering} is:

\begin{theorem}[See \cite{MR1161268, MR1239171}]  Assume that $f$ is an interval map which is $C^{1+Zygmund}$ and 
has non-flat critical points.  Then $f$ has
no wandering interval. Moreover, if $f$ is a $C^{2+zygmund}$ map with non-flat critical points, then there
exist $\kappa>1$ and $n_0\in \N$ such that 
$$|Df^n(p)|>\kappa$$
for every periodic point $p$ of $f$ of period $n\ge n_0$.
\end{theorem} 

Interestingly, it is not clear to what extent local symmetry around a critical point
is crucial. Indeed, consider a map of the form 
$$f(x)=\left\{ \begin{array}{rl}  x^\alpha +c  & \mbox{ for }x>0 \\
 				         x^\beta  +c & \mbox{ for }  x<0
				            \end{array}  \right. 
				            $$
with $\alpha\ne \beta , \alpha,\beta>1$ and $c$ real. It is not known whether such a map can have wandering intervals. For $\alpha=\beta$ the answer is no, due to the previous theorem.  For the case that 
$\alpha\ne \beta$ very little is known, except for the case that $\alpha=1<\beta$ and $f$ has
Feigenbaum-Coullet-Tresser dynamics, see  \cite{MR4152268} (and the proof of absence
of wandering intervals in that case follows a rather curious approach).

 The real bounds that go into the 
 proof of the absence of wandering intervals for real maps, 
 certainly inspired Dennis'  proofs of complex bounds which are crucial
 in his renormalisation theory. 
 
 As mentioned, a crucial ingredient in the proof of real bounds
 is Schwarzian derivative, or more generally the notion of {\em cross-ratio}. 
 A special cross-ratio inequality was used by Yoccoz to show that
smooth circle homeomorphisms with a unique non-flat critical point  cannot have wandering
intervals, see \cite{MR741080}. More general cross-ratio inequalities were 
then used in \cite{MR997312} for the interval case
and subsequently in \cite{MR968483} for circle endomorphisms, see  \cite[Sections IV.1 and IV.2]{MR1239171} for a discussion
of the connection between cross-ratio and Schwarzian derivative. 
In particular, the cross-ratio distortion arguments (i) suggest  the relevance of the 
 Poincar\'e metric on $(\C\setminus \R) \cup J$, which is the complex analogue of the 
 cross-ratio on a real interval $J$. Indeed, let $D_r(J)$ be the  set of points consisting 
 of the set of points with distance to $J$ of at most $r$  with 
 respect to the Poincar\'e metric on $(\C\setminus \R) \cup J$. 
 This set is often called a Poincar\'e disc, and is bounded by two arcs of the circles through $a,b$. 
 Using the Schwartz inclusion lemma, it then follows that if $f$ is (for example) 
 a real polynomial so that $f\colon J'\to J$  is a diffeomorphism and so that all critical values of 
 $f$ lie in $\R\setminus J$, then the component of $f^{-1}(D_r(J))$ intersecting 
 $J'$ is contained in $D_r(J')$.    This turned out to be a key ingredient to the proof of 
 his theorem on complex bounds for renormalisable maps, see Theorem~\ref{thm:complexbounds}.

\medskip 
Naturally, Dennis did ask himself whether there are analogues of his no wandering
domain theorem in the higher dimensional case in the smooth category. 
A partial answer to this question is given by Theorem~\ref{thm:withnorton} for toral diffeomorphisms
of Denjoy type.

\subsection{Conformal measures for rational maps} \label{sec:conf-meas-rat}

Soon after developing his no wandering theorem, Dennis  introduced the notion 
of {\em conformal $\delta$-measure} for a rational map $f$. 
This is a Borel probability measure $m$ 
so that 
$$m(fA)= \int_A |Df|^\delta \, dm , $$
for every Borel measurable set $A\subset \bar \C$ and 
where it is assumed that $\delta\ge 0$, see   \cite{MR730296}.
Dennis then showed that 
one can also construct conformal measures on the Julia set
analogous to what he  had done before in the setting of Kleinian groups, extending 
earlier work by Patterson: 
\begin{theorem}[Existence of conformal measures for rational maps, see  \cite{MR730296}] 
 For every rational map there exists a conformal measure. 
In the hyperbolic case, the exponent $\delta$ is positive and is equal to the Hausdorff dimension 
of the Julia set of the rational map. 
\end{theorem} 

The paper  \cite{MR730296} is also shows that for Dennis the theory of dynamical systems
is unified:  topological, smooth and ergodic aspects are all connected. Moreover, in his view the theory
of real and complex one-dimensional systems together with the theory of Kleinian groups
all should be viewed as highly interwoven.

\subsection{The $\lambda$-Lemma} 
One of Dennis' most used and cited papers on holomorphic dynamics is one in which he,
and his coauthors Ma\~n\'e and  Sad, proved that most maps are
{\em stable}. The main technical tool in that paper is the celebrated:

\begin{theorem}[$\lambda$-Lemma, see  \cite{MR732343}]\label{thm:lambdalemma} 
Let $A$ be a subset of $\overline  \C$, $\D$ the open unit disc and $i_\lambda\colon A\to \overline  \C$ a family
of maps so that
\begin{enumerate} 
\item  for each $z\in A$,  $\D\ni \lambda\mapsto i_\lambda(z)$ is  analytic; 
\item  $A\ni z\mapsto i_\lambda(z)$ is injective for each $\lambda\in \D$; 
\item $i_0=id$. 
\end{enumerate} 
Then every $i_\lambda\colon A\to \overline \C$ has a quasiconformal extension to a continuous map 
$i_\lambda\colon \overline A \to \overline  \C$, which for fixed 
   $\lambda$ is a topological embedding, and so that $\D\ni \lambda  \mapsto i_\lambda(z)$ is analytic for 
   each fixed $z\in \overline A$. 
\end{theorem} 

The proof of the $\lambda$-Lemma is surprisingly simple, and is based on the Schwarz's lemma
which states that any analytic map $\xi \colon \D \to \bar \C \setminus \{0,1,\infty\}$ 
is contracting w.r.t. the Poincar\'e metric on these sets. 
 Now consider the cross-ratio distortion of any distinct four points
in $A$.  Using that the cross-ratio distortion omits the values $0,1,\infty$ and this version of the  
Schwarz lemma, one obtains the above $\lambda$-Lemma.

In addition to the MRMT and the theory of quasiconformal homoeomorphisms,  the
 $\lambda$-Lemma has become one of the most widely used tools in the field of holomorphic dynamical 
 systems.\footnote{Independently, Lyubich proved an analogous result, see \cite{MR718838}.} 
 
 Later, jointly with Thurston, Dennis improved this $\lambda$-Lemma to show that one can extend $i_\lambda\colon A\to \overline \C$
 to a qc map  $i_\lambda\colon \overline \C \to \overline \C$, provided we restrict $\lambda$ 
 to a suitable disc $\D_0\Subset \D$ (here the choice of $\D_0$ is universal), 
 see \cite{MR857674}.

 \subsection{Density of stable maps} 

 The initial motivation for the $\lambda$-Lemma, and the  
  main purpose of the paper \cite{MR732343},  was to prove
 that stable maps are dense (within the space of rational maps). 
 As a first step towards proving this, the class of $J$-stable maps is considered. Here $f$ is called
 {\em $J$-stable} if for each $g$ near $f$ there exists a homeomorphism $h$ 
 of $J(f)$ to $J(g)$ so that $h\circ f=g\circ h$ on $J(f)$ and so that $J(g)$ depends continuously 
 on $g$   in the sense of Hausdorff distance between closed sets. $f$ is called {\em structurally stable}
 if the conjugacy holds on $\bar \C$.

Consider a family  $f_w(z)$ or rational maps depending on $w\in W$, where $W$ is a connected complex submanifold of $\C^{2d+1}$, and so that $(w,z)\to f_w(z)$ is analytic in $w,z$. Let $H(f)\subset W$ be the set of $w\in W$ for which there exists a neigbourhood $V$ with $w\in V\subset W$
so that each periodic point $p_w$ of $f_w$ depends analytically on $w\in V$ and so that 
their  multiplier satisfies either $\lambda(w)\ne 1$ for all $w\in V$ or $\lambda(w) \equiv constant$ for all $w\in V$.
 
 \begin{theorem}[$J$-stability, see  \cite{MR732343}] 
  $H(f)$ is  open and dense in $W$. Moreover, $f_w$ is $J$-stable if and only if $w\in W$ and 
 the conjugating homeomorphism $h_w$ can be taken to be analytic in $w$ and quasiconformal in $z$.
 \end{theorem}

 Analogous to the set $H(f)$, the authors introduce the set $C(f)\subset W$ of points $w$ for which  there is a neighbourhood $V$ with $w\in V\subset W$ so that each critical point $c_i(w)$ of $f_w$ depends analytically 
 on $w\in V$ and so that any critical relation relation $f^n_w(c_i(w))=f^m_w(c_j(w)))$ holds either for all $w$ in $V$ or 
 for none. 

   \begin{theorem}[Structurally stable maps are dense, see  \cite{MR732343}] 
  \label{thm:densstrstab} 
    $C(f)$ is an open and dense subset of $H(f)$.  Moreover, if $w\in C(f)$ then 
   $f$ is structurally stable. 
 \end{theorem}   
 
In the late 1960's Smale suggested that a similar result should hold for general 
smooth dynamical systems, but this turned out to be false (due to examples
by Newhouse and others).

 \subsection{Towards the Fatou conjecture: absence of line fields}

A map $f$ is said to be  {\em Axiom A}  (or {\em hyperbolic})  if there exist $\rho>1$  and $C>0$ so that 
$|(f^k)'(z)|>C \rho^k$ for all $k\ge 0$ and $z\in J(f)$. 
It is not hard to see that $f$ is Axiom A if and only if the periodic components of $F(f)$ are  superattractive or attractive and if the orbit of every critical point of $f$ is eventually contained 
in one of these components.

Fatou already stated the following:

\begin{conjecture}[Fatou Conjecture]   Each rational map
can be approximated by an Axiom A rational map of the same degree.
 \end{conjecture} 

No doubt Dennis tried to prove this conjecture but to this day nobody has succeeded in doing so. 
One of the main appealing properties of Axiom A maps  is that they are stable, 
provided they satisfy  some mild additional conditions, and  that their
dynamics is very well-understood. For example, for such a map 
Lebesque almost every initial point converges under iterates to a periodic attractor.

Amongst many other results in 
\cite{MR1620850}, Dennis,   together with McMullen,
shows that the above conjecture can be reduced to proving absence of  {\em measurable
invariant line fields} supported on Julia sets. 
Here a measurable line field on a forward invariant subset $K\subset J(f)$ of positive  
Lebesgue measure is a measurable function $z \mapsto \mu(z)$ on $K$. 
Here one can think of  $\mu(z)$ as a line  through $z$, and  invariance means that  $\mu(f(z)) = Df_z \mu(z)$. 
 
\begin{theorem}[A conditional proof of the Fatou conjecture, see  \cite{MR1620850}]  
Assume that {\em any} rational map which supports a measurable invariant line field on its Julia set
is a Latt\`es map. Then the above Fatou Conjecture holds. 
\end{theorem}

 \subsection{Monotonicity of entropy and the pullback argument} 

Another problem which was extensively studied in the early 1980's was whether
the topological entropy of the family $f_a\colon [0,1]\to [0,1]$, $a\in [0,4]$ defined by $f_a(x)=ax(1-x)$ is a monotone function in $a$. 
This problem was solved by several people independently and using different methods. For a history of this problem see \cite{MR4115082}.  Dennis' approach was particularly important
because it became a key ingredient in the proof of density of hyperbolicity within interval maps, see Theorem~\ref{thm:realfatou} below.

Monotonicity follows immediately from the following: 

\begin{theorem}[Monotonicity of entropy] \label{thm:rigiditycritfinite} 
Consider the family $f_a\colon [0,1]\to [0,1]$, $a\in [0,4]$ defined by $f_a(x)=ax(1-x)$.
Then no periodic orbit disappears as $a$ increases. 

If $f_a,f_{a'}$ are two such maps which are topologically conjugate and whose critical points are eventually
periodic, then $a=a'$.
\end{theorem} 

This theorem is non-trivial. Indeed, it is not known whether within the 
 family $x\mapsto x^d + c$ with $d>1$ fixed but not necessarily an integer, 
 bifurcations are monotone in $c$. Partial results, and monotonicity for 
 other non-trivial families of intervals maps,  are given in  \cite{MR4115082}.
One also has monotonicity within families of real polynomials of higher degree, namely 
each set of parameter for which the topological entropy is constant is connected. 
For the case of real cubic critical polynomials, see \cite{MR1736945} and for the general case 
of real polynomials with all critical points real, see 
\cite{MR3264762} and also \cite{MR3999686}.

 The rigidity statement in the second part of Theorem~\ref{thm:rigiditycritfinite} 
follows immediately from Thurston's famous theorem, see \cite{MR1251582}.
 The approach proposed by Dennis uses the {\em pullback argument}. 
 This argument  is formalised in the following theorem  and also applies
 to the setting of polynomial-like maps discussed below. 
 Let $P(f)$ be the closure of the forward iterates of critical points of $f$.

\begin{theorem}[Pullback argument] 
Let $h_0$ be a quasiconformal homeomorphism so that  
\begin{enumerate}
\item  $h_0(P(f_a))=P(f_{a'})$,  
\item there exists a qc map 
$h_1$ for which $f_{a'} \circ h_1 = h_0\circ f_{a}$ so that 
$h_1=h_0$ on $P(f_a)$ and  which is homotopic to 
$h_0$ rel. $P(f_a)$ and   
\item $h_0$ is a conformal conjugacy between $f_a$ and $f_{a'}$
near $\infty$ (and near their periodic attractor if they exist). 
\end{enumerate} 
Then there exists a qc homeomorphism $h$ so that $f_{a'} \circ h =  h\circ f_{a}$.
If $h$ is conformal near all periodic attractors (if they exist)
and if $f_a$ does not carry an invariant line field on $J(f_a)$
then  $h$ is conformal. 
\end{theorem} 
\begin{proof} By the  Homotopy Lifting Theorem, there exists a sequence of  
 homeomorphism $h_n$ so that 
$f_{a'} \circ h_{n+1} = h_n \circ f_{a} $ which are homotopic to $h_0$ rel. $P(f_{a})$. 
Since $f_a$ and $f_{a'}$ are conformal,   $h_{n+1}$ will have the same qc dilatation as $h_n$. 
Moreover, $h_{n+1}$ agrees with $h_n$ on a set $F_n$ so that $F_{n+1}\supset F_n$
so that $\cup F_n$ is dense in $\bar C$. Since the space of qc maps is compact, $h_n$ converges to a 
qc map $h$. If the Julia set of $f_a$ has zero Lebesgue measure (or, even more generally,  
does not carry invariant line fields) then $h$ is conformal. 
\end{proof} 

One can deduce  Theorem~\ref{thm:rigiditycritfinite} quite easily from the pullback argument, using
the {\em open-closed argument}. Indeed, assume by contradiction that the conclusion of Theorem~\ref{thm:rigiditycritfinite} 
is wrong. Then there exists two topologically conjugate post-critically  finite real quadratic maps $f_a$ and
$f_{a'}$ with $a\ne a'$. Choose $[a,a']$ {\lq}maximal{\rq} i.e., so that there exists no real $a''\notin [a,a']$
for which $f_{a''}$ is a real quadratic map which again is topologically conjugate to $f_a$.
The pullback argument implies that  $f_a$ and $f_{a'}$  are in fact quasiconformally conjugate. 
Let $h$ be a qc-conjugacy so that $f_{a'} \circ h =h\circ f_a$
and let $\mu$ be its Beltrami coefficient. Then the MRMT gives a (normalised) family of qc maps $h_t$ whose 
Beltrami coefficient is $t\mu$ for $|t|\le 1+\epsilon$, provided $\epsilon>0$ is small. 
A simple calculation then shows that each of the maps $g_t:=h_t\circ f_a\circ h_{t}^{-1}$
is again conformal. This is because $h_t$ sends the ellipse field determined by $t\mu$ to a field 
of circles,   and the invariance of the Beltrami coefficient implies that this ellipse field is preserved
by $Df$. Moreover,   $g_t$ depends analytically on $t$. That $g_t$ is  in fact is quadratic
follows from the degree of the map (and a suitable normalisation of $h_t$).  It follows that there exists $\epsilon>0$ 
so that for $s\in (a-\epsilon,a'+\epsilon)$ each  $f_s$ is quasiconformally conjugate to $f_a$. 
But this contradicts the maximality of the choice of $[a,a']$. 

This proof is very interesting because it makes it possible to reduce density of hyperbolicity 
with real quadratic maps to quasisymmetric rigidity in the unicritical setting, see Section~\ref{sec:realfatou}.
In fact, even if in the context of map with several critical points, quasisymmetric rigidity can be used
to derive density of hyperbolicity,  
see \cite{MR2533666, MR2342693, MR3336841}.

Another reason that makes  Dennis' proof of  Theorem~\ref{thm:rigiditycritfinite}  so interesting is that it 
also applies to the following setting,  introduced
by Douady and Hubbard \cite{MR812271}.

\begin{definition} Assume that $U$ and $V$ are simply connected
domains in $\C$. Then a holomorphic map
$F\colon U\to V$ is called
{\it quadratic-like}\index{quadratic-like}
if the closure of $U$ is contained
in $V$ and if there exists a unique critical point $c$ of $F$
such that $F$ restricted to $U\setminus \{c\}$
is a covering map of degree two onto $V\setminus \{F(c)\}$.
The subset
$$K(F)= \{z\in U : F^n(z)\in U  \mbox { for all } n\ge 0\}$$
is called  the {\it filled Julia set} of $F$.
\index{Julia set!filled}
\end{definition} 

Any quadratic map is quadratic-like (just
 take $U$ to be some very large disc). Moreover,
by the so-called Straightening Theorem, see  \cite{MR812271},  any quadratic-like
map  is quasiconformally conjugate to a quadratic map.

\medskip
One step in the renormalisation theory developed by Dennis (and also
in subsequent developments) is to show that certain iterates of a given map
have quadratic-like restrictions  $f^n\colon U_n\to V_n$ 
with the additional property that the modulus 
$\mod(V_n\setminus U_n)$ \, is bounded from below uniformly in $n$. 
Such bounds are called {\em a-priori bounds} or {\em complex bounds}.

\subsection{Renormalisation theory for interval maps} 
 
Consider the family $f_a(x)=ax(1-x)$.  A  simple computer simulation shows 
 that this family of maps undergoes a period doubling bifurcation 
from period $2^n$ to period $2^{n+1}$ at some parameter $a_n$. That these parameters
$a_n$ are in fact unique (and increasing) can be deduced from a result similar to 
Theorem~\ref{thm:rigiditycritfinite}.

%\bigskip
% \begin{definition} A interval map $f$ is {\em renormalisable}
% if there exist $p>0$ and  an interval $J$ around a critical point of $f$
% so that $J,\dots,f^{p-1}(J)$ have disjoint interiors and so that 
% $f^p(J)\subset J$. 
% \end{definition} 
%
%If $f$ is unimodal, then $f^p|J$ is again a unimodal map
%and it makes sense to require that $p$ is minimal and $J$ is maximal
%with the above properties.  
%
%The map $$R(f)=f^2|J \mbox{ rescaled }$$
%is called the renormalisation of $f$. If $R(f)$ is again renormalisable, 
%then we say that $f$ is twice renormalisable. Similarly, $f$ is called {\em  infinitely renormalisable}
%if this process can be repeated infinitely often. If the corresponding integers $p_1,p_2,\dots$ are
%all bounded by some number $p<\infty$ then $f$ is called infinitely renormalisable
%if bounded type. 
%
% \section{Renormalisation} 

One of the reasons that iterations of interval maps  attracted so much attention
from the late 1970's  was the observation by Feigenbaum and independently 
by Coullet and Tresser of  metric {\em universality}  
within a wide class of such families. Namely,  it turns out that
the parameters $a_n$ converge to some limit value $a_\infty$
at a particular rate $\delta>1$:
$$\dfrac{a_{n-1}-a_{n-2}}{a_{n}-a_{n-1}} \to \delta=4.669201...$$
Remarkably, if one takes some other 
family such as $f_a(x)=a\sin(\pi x)$ then the corresponding rate is the same! 
Moreover,  the map $f_{a_{\infty}}$ has an invariant Cantor set
and the scaling structure of this Cantor set also displays metric universality. 

Feigenbaum, Coullet and Tresser already suggested a mechanism 
which would be responsible for this universality. The key idea is 
renormalisation. Indeed, there exists a (real) parameter interval
$[u_1,v_1]\ni a_\infty$ so that for each $a\in [u_1,v_1]$
there exists an interval $J^1_a\ni 1/2$ so that $f_a^2(J^1_a)\subset J^1_a$
and so that $f_a(J^1_a)$ and $J^1_a$ have disjoint interiors. 
Such maps are called {\em $2$-renormalisable}.  Note that $1/2$ is the critical point
of $f_a$. 
It turns out that there exists an interval $[u_2,v_2]\subset [u_1,v_1]$
so that for each $a\in [u_2,v_2]$ the map $f^2_a|J^1_a$ (rescaled)
is again $2$-renormalisable.  In other words,  there exists an interval
$J^2_a$ with $1/2\in J^2_a\subset J^1_a$ so that $f_a^4(J^2_a)\subset J^2_a$
and so that $J^2_a,\dots,f^3_a(J^2_a)$ have disjoint interiors. 
Moreover, $J^2_a,f^2_a(J^2_a)\subset J^1_a$ and $f(J^2_a),f^3_a(J^2_a)\subset f(J^1_a)$.
So these maps are twice $2$-renormalisable. Continuing like this, 
for each $k$ there exists an interval  $[u_k,v_k]$ so that for each $a\in [u_k,v_k]$
the map $f_a$ is $2^k$-renormalisable. For $a_\infty \in \cap [u_k,v_k]$
the set 
%$$\bigcap_{k\ge 0}\,\, \bigcup_{i=0}^{2^{k-1}} f^i_a(J^k_a)$$
$$\Lambda_{a_\infty} :=\bigcap_{k\ge 0}\,\, \left( J_a^j\cup \dots \cup  f^{2^k-1}_a(J^k_a) \right)$$
is a Cantor set.  

To formalise this one can define, near the limit map $f_{a_\infty}$, the renormalisation operator 
$$R(f)=f^2|J \mbox{ rescaled }$$
where $J$ is the maximal interval of renormalisation of period two, i.e. the 
maximal interval  so that $f^2(J)\subset J$
and so that $f(J)$ and $J$ have disjoint interiors.  This operator 
is well-defined for all maps which are at least once $2$-renormalisable.  

More generally one has the following

 \begin{definition} A interval map $f$ is {\em renormalisable}
 if there exist $p>0$ and  an interval $J$ around a critical point of $f$
 so that $J,\dots,f^{p-1}(J)$ have disjoint interiors and so that 
 $f^p(J)\subset J$. The operator $R(f)=f^p|J \mbox{ rescaled }$ is then 
 called the renormalisation operator. See figure \ref{figure1}.
 \end{definition} 
 
%%%%%%%%%%%%%%%%%%%%%%%%%%%%%%%%%%%%

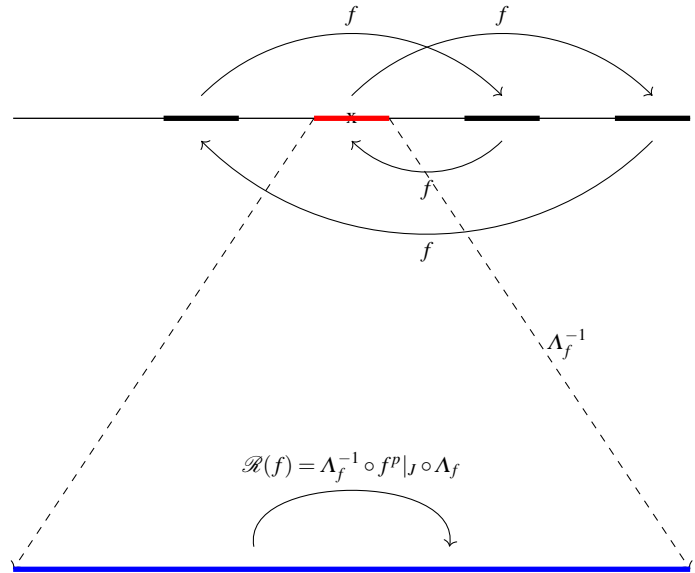
\begin{figure}[ht]
 \begin{center}{
\begin{tikzpicture}[out=45,in=135,relative]
%\def\a{0.7}\def\b{0.4}
% Intervals
\draw[line width=0.5pt] (0,0) to[line to] node {x} (9,0);
\draw[red,line width=2pt] (4,0) to[line to] (5,0);
\draw[line width=2pt] (2,0) to[line to] (3,0);
\draw[line width=2pt] (6,0) to[line to] (7,0);
\draw[line width=2pt] (8,0) to[line to] (9,0);
% Arrows
\draw[->] (4.5,0.3) to[curve to] node[above] {$f$} (8.5,0.3);
\draw[->] (8.5,-0.3) to[curve to] node[below] {$f$} (2.5,-0.3);
\draw[->] (2.5,0.3) to[curve to] node[above] {$f$} (6.5,0.3);
\draw[->] (6.5,-0.3) to[curve to] node[below] {$f$} (4.5,-0.3);
% Rescaling
\draw[->,dashed] (4,0) to[line to] node[left] {} (0,-6);
\draw[->,dashed]  (5,0) to[line to] node[right] {$\Lambda_f^{-1}$} (9,-6);
\draw[blue,line width=2pt] (0,-6) to[line to] (9,-6);
\draw[->,out=100,in=80] (3.2,-5.7) to[curve to]  node[above] {$\mathcal{R}(f)=\Lambda_f^{-1}\circ f^p|_{J}\circ \Lambda_f$} (5.8,-5.7);
\end{tikzpicture}
\caption[figure1]{\label{figure1} Renormalising a unimodal map. Here, $J$ is the red interval, $p=4$, and $\mathcal{R}(f)$ is simply $R(f)$ rescaled by the affine map that takes the blue interval onto the red interval.}
}
\end{center}
\end{figure}

%%%%%%%%%%%%%%%%%%%%%%%%%%%%%%%%%%%%

 If $f$ is unimodal, then $f^p|J$ is again a unimodal map
and it makes sense to require that $p$ is minimal and $J$ is maximal
with the above properties.    If $R(f)$ is again renormalisable, 
then we say that $f$ is twice renormalisable. Similarly, $f$ is called {\em  infinitely renormalisable}
if this process can be repeated infinitely often. If the corresponding integers $p_1,p_2,\dots$ are
all bounded by some number $P<\infty$ then $f$ is called infinitely renormalisable
of bounded type.

It turns out that the renormalisation conjectures of  Feigenbaum and 
Coullet \& Tresser of metric universality follow from 
(i) the existence of a fixed point $\psi$ of the operator $R$, (ii) 
that the spectrum of the operator $DR\psi$ lies off the unit circle
and (iii) that $DR\psi$ has a unique expanding eigenvalue. 
The universal parameter scaling constant $\delta$ is equal to this 
expanding eigenvalue.  The universal dynamical scaling structure of $\Lambda_{a_\infty}$ follows
from the largest contracting eigenvalue of $DR$. 

The universality from the Feigenbaum-Coullet-Tresser conjectures
then follows from the fact that for any family which crosses
the stable manifold of $R$ transversally, the parameter scalings $\delta$
and the dynamical scaling can be obtained from the spectrum 
of $DR(\psi)$. 

The existence of the fixed point $\psi$, and an analysis of the spectrum
of the linear map $DR$  were established by Lanford before Dennis started
working on the renormalisation conjectures. 
This was done in part using careful rigorous computer estimates. 
However, there were three limitations to Lanford's results. 

Firstly, Lanford's proof did not establish which maps are contained  in the stable manifold of $R$.
In other words, what remained unclear whether any $2^\infty$-infinitely renormalisable unimodal maps
(with a quadratic critical point) would be in the stable manifold of 
of the period doubling operator. Secondly,  Lanford's proof did not establish 
a conceptual proof of why this result was true. 
Finally, Lanford's proof also did not cover a more general situation of maps
which are infinitely renormalisable of bounded type, but only 
of constant type $p_1=p_2=\dots$.  
%
%Indeed, for each $p$ there are some maps which are $p$-renormalisable, 
%i.e. there exists an interval $J$ so that $f^p(J)\subset J$ so that $J,\dots,f^{p-1}(J)$
%are pairwise disjoint. (To set this up properly, one also needs to specify
%the ordering of these intervals). Thus there exists maps 
%which are infinely renormalisable of type 
%$p_1\times p_2\times p_3\times \dots$.  Such maps are called of {\em bounded type} 
%if there exists $p$ so that $p_i\le p$ for all $i\ge 0$. 

The huge result which Dennis managed to obtain is the following:

\begin{theorem}[Renormalisation for unimodal interval maps, see \cite{MR1184622} and  also \cite{MR1239171}]
\label{thm:renormalisation} 
There exists a Cantor set $\mathbb{K}_p$ of infinitely renormalisable maps of bounded type $\leq p$, 
of real analytic unimodal maps, which form an invariant subset for the 
corresponding renormalisation operator. The renormalisation operator acts on $\mathbb{K}_p$ as a full shift on finitely many symbols.
Each real analytic unimodal map with a quadratic critical point which is 
infinitely renormalisable map and of bounded type
is in the stable manifold of the corresponding renormalisation operator. 
\end{theorem}

After this,  other proofs of the renormalisation conjectures appeared. 
McMullen \cite{MR1312365, MR1401347}
and Avila and Lyubich \cite{MR2854860}  
gave easier proofs that the stable manifold of the renormalisation operator
contains all the relevant infinitely renormalisable maps using {\lq}towers{\rq}
respectively based on a Schwarz Lemma. 
These proofs additionally  give that  under renormalisation infinitely renormalisable 
maps of bounded type converge with an exponential rate to the above Cantor
set of  infinitely renormalisable maps of bounded type. 
Moreover, Lyubich \cite{MR1689333}  gave a conceptual proof showing that
the renormalisation operator has a unique expanding eigenvalue. 

The above proofs require that the maps are real analytic. 
In the $C^r$ context, these theorems also go through, see \cite{MR2259245}. 
An alternative approach to extend the renormalisation theory for real analytic maps
to smooth maps is via asymptotically holomorphic maps, see 
\cite{clark_de_faria_van_strien_2022}.

In the multimodal setting, the renormalisation picture is not complete yet. 
In this setting, the conjecture could be 
\begin{enumerate}
\item Topologically conjugate mappings converge exponentially quickly under renormalisation.
\item The stable manifolds of renormalisation are smooth.
\item The transverse directions to the stable manifolds are exponentially expanded by renormalisation.
\end{enumerate} 
Part 2 of this conjecture has been proved in \cite{Clark-Strien}. For the case of bounded combinatorics, 
see \cite{MR4045962}.

\subsection{Real and complex bounds} 
The first step towards proving Theorem~\ref{thm:renormalisation} is to show that 
the space of renormalisable maps is compact. To do this, Dennis 
established apriori bounds, first in the real and then in the complex setting: 

%\begin{theorem}[Real bounds] 
%\end{theorem} 

\begin{theorem}[Real Bounds] 
Let $f$ be a real analytic map which is infinitely renormalisable and of bounded type.  
Then the $C^2$ norm of the maps $R^nf$ is
uniformly bounded.  
\end{theorem} 

The first part of this theorem shows that $R^nf$ is a composition of a quadratic map
and a map $g$ whose non-linearity is bounded from above. The main 
ingredient he used for this is the (real) Koebe Principle  and the {\em smallest interval argument} discussed above. 
Indeed, let $J$ be the first renormalisation interval and that $f^p(J)\subset J$ with $p\ge 1$ minimal.  
Then the intervals $J,\dots,f^{p-1}(J)$ are disjoint and one among them, let us say $f^i(J)$, is the smallest. This means that, unless
$f^i(J)$ is one of the two extreme intervals in this collection,  
the interval $f^i(J)$ is contained in an interval $T=[f^l(J),f^r(J)]$ which has the property 
that both components of $T\setminus f^i(J)$ are not small compared to $f^i(J)$. 
This simple idea can be used to obtain {\em Koebe space}, namely that the map 
$f^{p-1}\colon f(J)\to f^p(J)\subset J$ extends to a diffeomorphism with range 
$T'\supset J$ so that $J$ is well-inside $J$. Using the real Koebe Principle 
the map $f^{p-1}\colon f(J)\to f^p(J)$ then has bounded non-linearity.

Extending this argument, and using the pullback argument from above, 
one then obtains one of the key steps: 

\begin{theorem}[Quasisymmetric Rigidity in the Renormalisable Case] 
Let $f,g$ be two infinitely renormalisable real-analytic maps of bounded type
and with quadratic critical points.  Then if $f,g$ 
are topologically conjugate they are in fact quasisymmetrically conjugate. 
\end{theorem} 

To obtain Theorem~\ref{thm:renormalisation} Dennis needed to extend 
the real maps $f^p\colon J\to J$ to quadratic-like maps and to obtain 
compactness with the space of such maps: 

\begin{theorem}[Complex Bounds] 
\label{thm:complexbounds} 
Let $f$ be an infinitely renormalisable real-analytic unimodal map of bounded type $p_i\le p$ for all $i\ge 0$ and with quadratic critical point.  Then, for every $n$ sufficiently large, $R^nf$ extends to a quadratic-like map
$F\colon U\to V$ so that  the modulus 
of $V\setminus U$ is bounded by some number $\rho>0$ which does not depend on $f$ but only the upper bound $p$.  
\end{theorem} 

Moreover, any limit of the sequence $\{R^nf\}$ has a complex analytic extension 
which is in the so-called Epstein class.

Real bounds were proved in a much more general context, see (in increasing generality)
 \cite{MR1279474,MR1963798, MR2083467}. Complex bounds for real unicritical 
 maps were proved in \cite{MR1637647,MR1488251} and for multicritical maps in 
 (in increasing generality) in  \cite{MR1862823, MR3687213}.
Complex bounds do not hold for general (non-real) quadratic maps: there are examples
of infinitely renormalisable quadratic maps for which no modulus bounds as in the above theorem 
hold.  On the other hand, for non-renormalisable polynomials maps (with only hyperbolic periodic points) 
one does have complex bounds, see 
\cite{MR2533666} and 
\cite{MR4446272}.  An important ingredient in the latter developments
is the quasi-additive lemma by Kahn and Lyubich \cite{MR2480612}. 
This lemma was also used to treat some maps which are infinitely renormalisable, 
see \cite{MR2423310,MR2508259}.  
It is not known how to extend complex bounds to the case of general rational maps, 
as in general it is not clear how to construct an initial puzzle partition.

\subsection{Riemann surface laminations and the non-coiling lemma} 

To complete the proof of Theorem~\ref{thm:renormalisation}, Dennis introduced a new tool, namely 
his  {\em non-coiling principle} and his {\em almost geodesic principle}. To explain this, 
 consider a  qc conjugacy $H$ between 
$F_0$ and $F_1$. Its Beltrami coefficient $\mu_H=\bar \partial H/ \partial H$ its invariant 
under $F_0$. It follows that the family of qc maps $H_t$ associated to 
$\mu_t=t\mu_H$ (coming from the MRMT) defines a family of quadratic like maps $F_t=H_t\circ F_0\circ H_t^{-1}$
connecting $F_0$ to $F_1$. This is called a {\em Beltrami path} between $F_0$ and $F_1$.  
Dennis' {\em almost geodesic principle}  shows that  the Beltrami path corresponding to an {\em almost extremal vector} does not coil: if the tangent Beltrami vector is almost extremal then the Beltrami path remains almost a geodesic for a long (but a priori fixed) time.

To show  that the renormalisation operator is (weakly) contracting
he then argued as follows. From the complex bounds discussed in the previous subsection, he obtained that  there 
exists a compact space $\mathcal K$, so that if we take an arc connecting two conjugate maps which are infinitely renormalisable  and of bounded type, then after $n$ renormalisations this arc is mapped in $\mathcal K$.
Here $n$ depends on the choice of the chosen maps, but $\mathcal K$ does not.  
%Dennis introduced a distance on the space of quadratic-like mappings: if $F_i\colon U_i\to V_i$ are  quadratic-like 
%mappings then 
%$D_{T}(F_0,F_1) = \inf_H \log \dfrac{1+||\mu_\phi ||}{1-||\mu_\phi || }$
%where $H$ ranges over all qc conjugacies between $F_0$ and $F_1$. The next ingredient he then uses  
Now take an almost geodesic path between two maps $F_0,F_1$. Extend this path to a geodesic path 
between two maps $\tilde F_0,\tilde F_1$ which are extremely far apart. Now apply renormalisation. 
For $n$ sufficiently large, the renormalisations $R^n(\tilde F_0),R^n(\tilde F_1)$ are in $\mathcal K$ and 
so not far apart. 
But then, by the {\em almost geodesic principle},  the renormalisations $R^n(F_0),R^n(F_1)$ are extremely close. 

To make all this work, Dennis had to consider germs of quadratic-like maps.  
For this reason he considered inverse limits of the quadratic-like maps $F_i\colon U_i\to V_i$ which led to the study of \emph{Riemann surface laminations}. This is not the place to go into a full description of the beautiful theory of such objects, but let us at least say a word or two. Roughly speaking, a Riemann surface lamination (RSL) is a space akin to a foliated space in which the chart domains are homeomorphic to $D\times T$, where $D\subset \mathbb{C}$, is a disk and the transversal $T\subset \mathbb{R}$ is typically a Cantor set (or an interval), and the chart transitions are holomorphic along the horizontal leaves. In the present context, the main example is the following. Let $F\colon U\to V$ be a a quadratic-like map, let $K(F)\subseteq \mathbb{C}$ be its filled-in Julia set, which we assume to contain the critical point of $F$ (so that it is connected), set $W=V\setminus K(F)$, and consider the inverse system of holomorphic covering maps:
\[
\cdots \to F^{-n-1}W \to F^{-n}W\to \cdots \to F^{-1}W \to W\ .
\]
The inverse limit space $\mathcal{L}(F)$ of this system is a fibration over $W$, the fiber above each $x\in W$ being a Cantor set (the binary Cantor set at the end of the tree giving the full backward orbit of $x$). From this it follows that $\mathcal{L}(F)$ is an RSL in a natural way. 

The inverse limit map $F_\infty: \mathcal{L}(F) \to \mathcal{L}(F)$ is invertible and acts properly discontinuously on $\mathcal{L}(F)$, and the quotient $X_F=\mathcal{L}(F)/F_\infty$ is a \emph{compact} RSL. In \cite{MR1184622}, Dennis defined a deformation space or \emph{Teichm\"uller space} of $X_F$ (and more general RSLs) in such a way that every Beltrami path between two quadratic-like maps $F_0$ and $F_1$ as above can be lifted to a Beltrami path between the corresponding laminations $X_{F_0}$ and $X_{F_1}$, and all deformations are encoded in this fashion. What makes this possible is the fact that the Julia set of an infinitely renormalisable quadratic-like map with complex bounds does not carry any non-trivial quasi-conformal deformations. 
Dennis then proved the non-coiling principle and the almost-geodesic principle at the level of laminations, transporting the resulting contraction of the Teichm\"uller distance downstairs, at the level of maps.

%Dennis was able to derive from the {\em Almost Geodesic Principle} that the renormalisation operator 
%is weakly contracting. 

We will not go further into Dennis proof of this tour de force (a full description can be found in \cite[Chapter VI]{MR1239171}, in addition to the papers \cite{MR1184622} and \cite{MR1215976}). 
After all, as mentioned,  this last step was improved in subsequent proofs which 
show that the invariant Cantor set of $R$ attracts other maps with an {\em exponential rate}, see McMullen \cite{MR1312365, MR1401347} and Avila and Lyubich \cite{MR2854860}.  
Still, the reader who wants to learn more about RSLs should consult the elegant survey written by Ghys in \cite{MR2017612}. 

\subsection{Renormalisation theory for circle maps} 

%\textcolor{red}{Edson to add text here and add more references (incl. PhD)} 
%\cite{MR1728375}, \cite{MR1711394}. 

While working on the renormalisation problem for unimodal maps, Dennis was well aware that similar experimental discoveries to those made by Feigenbaum and Coullet-Tresser had been made by physicists concerning circle homeomorphisms having a single non-flat critical point (of power-law type). The topological classification of such maps had been accomplished by Yoccoz in \cite{MR741080}. 

Dennis also knew that, in the circle context, Lanford had formulated a renormalisation conjecture akin to the one for unimodal maps, using the language of commuting pairs. Dennis then suggested to EdF, as a thesis problem, to adapt his holomorphic ideas to the case of such critical circle maps. The key step was to find an analogue of quadratic-like maps in the context of critical circle maps.  This was accomplished in \cite{MR2688469} (see also \cite{MR1709428}) with the notion of \emph{holomorphic commuting pair} (inspired in part by a computer picture drawn by H. Epstein), alongside a proof of complex bounds (as well as a pull-back argument) for such objects,  assuming the rotation number of the underlying critical circle map to be of bounded type, and also that the circle maps belonged to a special class of maps known as \emph{Epstein class}. The necessary real bounds had already been established by Herman (unpublished manuscript, but see \cite{MR3839606}) and Swiatek \cite{MR968483}. The bounded type assumption was removed by Yampolsky \cite{MR1677153}, still assuming the Epstein property. The latter was finally removed by EdF and Welington de Melo in \cite{MR1711394}. 

For unicritical circle maps, the fact that, under suitable full-measure conditions on the rotation number, exponential convergence of renormalisations leads to $C^{1+\alpha}$ rigidity was established in \cite{MR1728375} (counterexamples to this ansatz for rotation numbers in a special zero-measure class were constructed in the same paper). The analogous conditional statement obtained replacing $C^{1+\alpha}$ by $C^1$ holds under no restriction on the rotation number (other than being irrational), as shown by \cite{MR2308853}. The exponential convergence of renormalisations for \emph{real-analytic} unicritical circle maps with bounded type rotation number was proved in \cite{MR1711394}. Using the concept of parabolic renormalisation, Yampolsky \cite{MR1985030} was able to remove the bounded type hypothesis, and in fact proved that the renormalisation operator attractor is globally hyperbolic in the analytic context. 
In the larger space of $C^3$  unicritical circle maps, exponential convergence towards the attractor was first proved by Guarino in his thesis under de Melo -- see  \cite{MR3646874} -- assuming rotation numbers of bounded type only. The bounded type hypothesis was later removed in \cite{MR3843373}, at the cost of assuming the maps to be $C^4$. 

In recent years, considerable work has been done to extend these rigidity, universality and renormalisation convergence results to multicritical circle maps -- see for instance \cite{MR4458800} and references therein.
An important step towards this goal is to first establish the \emph{quasisymmetric rigidity} of such maps -- this was accomplished in \cite{MR3812112} for multicritical circle maps with critical points having arbitrary (real) power-law criticalities. 
For this and much more about multicritical circle maps, see \cite{MR4472818}.

\subsection{The Fatou conjecture in the real setting} \label{sec:realfatou}

As mentioned, Fatou conjectured that each rational map can be  approximated by an Axiom A rational map of the same degree. This question is still wide open, even in the quadratic case. 
However,  in the setting of real maps the corresponding result has been answered completely:

\begin{theorem}[The Real Fatou Conjecture] \label{thm:realfatou}   Each real polynomial can  can be approximated by a real Axiom A polynomial of the same degree.
 \end{theorem} 

In the setting of real quadratic polynomials, the real Fatou conjecture was proved independently by Lyubich \cite{MR1459261}
and Graczyk \& Swi\'atek \cite{MR1469316,MR1657075}. 
The setting of real polynomials of higher degree $d>2$ was solved (using entirely different tools) by Kozloski, Shen and van Strien, 
see \cite{MR2335796} and \cite{MR2342693}. In the non-real non-renormalisable 
case see also \cite{MR2533666} and 
\cite{MR4446272}.  An important ingredient in the latter developments
is the quasi-additive lemma by Kahn and Lyubich \cite{MR2480612}. 
This lemma was also used to treat some maps which are infinitely renormalisable, 
see \cite{MR2423310,MR2508259}.

Although the proofs of these results are not due to Dennis, he played an important role 
in them. Indeed, his work suggested that to prove density of hyperbolicity 
that it would be enough  to prove the following

\begin{theorem}[Quasisymmetric Rigidity in one-dimensions]
If $f,g$ are topologically conjugate real polynomials with only real critical points, 
and all their critical points are quadratic, then these maps are quasisymmetrically conjugate. 
\end{theorem} 

Dennis'  renormalisation theory for infinitely renormalisable unimodal (unicritical) maps of bounded type, 
relied on this result (in this setting).  In the quadratic case, 
the above theorem is due to  \cite{MR1459261} and Graczyk \& Swi\'atek \cite{MR1469316,MR1657075}. 
Their proof relied on the property that, in this setting, the moduli of certain annuli tends to infinity. 
This growth of moduli is a deep and subtle result, but this does not hold for unimodal maps 
with a degenerate critical point nor for mulimodal maps with non-degenerate critical points.  
So for the general case a different approach is needed. The approach by Kozlosvki, Shen and van Strien in  \cite{MR2335796} uses the enhanced nest, which is a particular choice of a sequence of puzzle pieces
that turn out to have Koebe space.  An introductory survey on this technique can be found 
in \cite{MR4446272}. The most general quasisymmetric result
is contained in joint work of Clark and van Strien, see \cite{https://doi.org/10.48550/arxiv.1805.09284}. 

\begin{table}[h!]
\begin{center}
\begin{tabular}{||c|c||} 
\hline 
{}&{}\\
{\bf{Kleinian Groups}} & {\bf{Iterated Analytic Maps}}\\
{}&{}\\
\hline 
{}&{}\\
{Kleinian group $\Gamma$} & {Holomorphic map $f$}\\
{}&{}\\
\hline 
{}&{}\\
{$\Gamma$ finitely generated} & {$f$ rational map}\\
{}&{}\\
\hline 
{}&{}\\
{$\Gamma$ is Fuchsian} & {$f$ is a Blaschke product}\\
{}&{}\\
\hline 
{}&{}\\
{Domain of discontinuity $\Omega(\Gamma)$} & {Fatou set $\mathcal{F}(f)$}\\
{}&{}\\
\hline 
{}&{}\\
{Limit set $\Lambda(\Gamma)$} & {Julia set $J(f)$}\\
{$\Lambda(\Gamma)\neq \void$}& {$J(f)\neq \void$}\\
{}&{}\\
\hline 
{}&{}\\
{$\Omega(\Gamma)$ has either $0,1,2$ or} & {$\mathcal{F}(f)$ has either $0,1,2$ or}\\
{infinitely many components}&{infinitely many components}\\
{}&{}\\
\hline 
{}&{}\\
{Either $\Lambda(\Gamma)=\widehat{\mathbb{C}}$ } & {Either $J(f)=\widehat{\mathbb{C}}$ }\\
{or $\Lambda(\Gamma)$ has empty interior}& {or $J(f)$ has empty interior} \\
{}&{}\\
\hline
{}&{}\\
{Ahlfors finiteness theorem} & {Sullivan's no-wandering-domains theorem}\\
{}&{}\\
\hline
{}&{}\\
{Bers area theorem} & {Shishikura's bound on the number}\\
{}&{of non-repelling periodic cycles}\\
{}&{}\\
\hline
{}&{}\\
{Mostow’s rigidity theorem} & {Thurston's uniqueness theorem on}\\
{}&{post-critically finite rational maps}\\
{}&{}\\
\hline 
{}&{}\\
{\;Patterson-Sullivan measures on $\Lambda(\Gamma)$\;} & {Sullivan's conformal measures on $J(f)$}\\
{}&{}\\
\hline
{}&{}\\
{The quotient manifold $\mathbb{H}/\Gamma$} & {Lyubich-Minsky lamination}\\
{}&{}\\
\hline
{}&{}\\
{Geometrically finite groups} & {Are hyperbolic rational}\\
{with no cusps are dense} & {maps dense?}\\
{}&{}\\
\hline
\end{tabular}
\end{center}
\caption{\label{tableone} Some entries in Sullivan's dictionary}
\end{table}

\subsection{Sullivan's quasisymmetry rigidity programme} 
Even though there is no analogue of the MRMT in the real setting, one
of the insights of Dennis was that quasisymmetric rigidity should still be a very powerful 
tool in addressing questions about the topological structure of conjugacy classes 
of interval maps. For example,  whether such a conjugacy class is a connected manifold. 
This insight turned out to be justified.  Indeed, let $\mathcal A^{\underline \nu}$ be the 
space of real analytic maps with precisely $\nu$ real critical points $c_1<\dots<c_\nu$ 
of order $\ell_1,\dots,\ell_\nu$.  The following theorem was shown by Clark and van Strien  
\cite{Clark-Strien}.
\begin{theorem} Let $f\in \mathcal A^{\underline \nu}$. Then the space $\mathcal T_f$ of real analytic maps 
in $\mathcal A^{\underline \nu}$ which are topologically conjugate to $f$ forms an analytic manifold. This manifold is 
connected and simply connected. 
\end{theorem} 
 
 This theorem extends results of Avila-Lyubich-de Melo \cite{MR2018784} 
for the quasi-quadratic unimodal case and of Clark \cite{MR3255432} for the more general unimodal case. Their methods fail in the case where there are several critical points. For this reason, the notion of  {\em pruned Julia set} 
is introduced in \cite{Clark-Strien}. This set is a version of the Julia set (but pruned) but depends on where one
{\lq}prunes{\rq}. A pruned Julia set can be defined for 
each real analytic map $f$. The real analytic map $f$, together with its pruned 
Julia set, define a real analytic {\em external map} of the circle {\em with discontinuities}. 
Using this external map, one can construct a  {\em a pruned polynomial-like} complex extension
of the real analytic map. Finally, from all this one is able to show that topological conjugacy classes are connected (something  which was not even known in the general unimodal setting). Even more,  this space is contractible
and forms an analytic manifold.  

\section{Sullivan's dictionary}\label{sullivandict}

Dennis' wide-range view of Mathematics allows him to draw fruitful analogies between different theories, leading to several conjectures on either side. A case in point is what is now known as the \emph{Sullivan dictionary} between the theory of Kleinian groups (in dimension $n=3$) on one side and the theory of iterated holomorphic maps on the other side. A sample of entries in this dictionary is shown in Table \ref{tableone}. Note that the last entry in the table has a question mark: that is none other than the famous Fatou Conjecture, widely regarded as the main classical open problem about the dynamics of rational maps. 

%The analogy, however, is not always perfect. 
Not all meaningful analogies, however, deserve to be in the dictionary.
For instance, a famous conjecture by Ahlfors in the 1960's stated that the limit set of a finitely generated Kleinian group is either the entire sphere or else has zero Lebesgue measure. As we mentioned in the beginning of section \ref{sec:kleiniangroups}, this is now a theorem, thanks to the combined efforts of several mathematicians. The final piece of the puzzle was laid down by Canary \cite{MR1166330}, building primarily on previous works by Thurston and Bonahon -- see for instance \cite{MR2355387} for a description of the whole story, and references therein. The corresponding statement for iterated holomorphic maps -- to wit, that the Julia set of a rational map is either the entire sphere or else has zero Lebesgue measure -- was thought for a long time to be true, until X.~Buff and A.~Chéritat \cite{MR2950763} found an example of a quadratic polynomial whose Julia set has positive measure. 

As Dennis himself has explained to us, he never put the Ahlfors conjecture in the dictionary because -- after working on the problem for about 13 months in the late seventies and exhausting all available ergodic arguments that would have solved it --  he came to the conclusion that no one would be able to prove it using what was currently known about finitely generated Kleinian groups. 
After proving his finiteness theorem, what Ahlfors really wanted to know was under which conditions the limit set of a Kleinian group could support non-trivial quasiconformal deformations. He asked the question about the Lebesgue measure of the limit set in the finitely generated case because, if the measure indeed turned out to be zero in that case, there would be no such deformations. Thus, Dennis realized that the real question was not the measure zero question, but rather to describe, if any, the quasiconformal deformations on the limit set. He was able to prove a very general result that states that, given a Kleinian group $\Gamma$ and any $\Gamma$-invariant subset $E\subset \Lambda_\Gamma$ of its limit set, there are quasiconformal deformations of $\Gamma$ supported in $E$ if and only if  there are \emph{positive measure wandering sets} inside $E$, and when this happens, the space of such nontrivial deformations is infinite dimensional. In particular, since for finitely generated groups the space of deformations is a-priori known to be finite-dimensional, that are no quasiconformal deformations supported on $\Lambda_\Gamma$ when $\Gamma$ is finitely generated. This holds even if the limit set happens to be the whole sphere. 
This absence of invariant line fields supported in the limit set is stated in Dennis'  Theorem \ref{noinvariantlinefields}. After more than 40 years, the corresponding  statement for rational maps{\footnote{To wit, that a rational map is either a Lattès example or else carries no invariant line fields in its Julia set.}} remains an open problem. And it is a fundamental problem: indeed it is possible to prove that if the statement for rational maps is true, then so is the Fatou conjecture. Thus, the question of absence of invariant line fields certainly deserves its place as an entry in the Sullivan dictionary (albeit being conspicuously absent from Table \ref{tableone}). 

Over the years, the Sullivan dictionary has continued to inspire new results. A recent example is provided by the work of Hee Oh. Working on the Kleinian side of the dictionary, in collaboration with Margulis and Mohammadi \cite{MR3261636}, she examined closed geodesics and holonomies for hyperbolic $3$-manifolds. She was then asked by Dennis himself about an analogue of her results for rational maps. This resulted in her paper \cite{MR3639596} in collaboration with Winter, in which they establish estimates on the number of primitive periodic orbits of a hyperbolic rational map. 

\section{Final words}

Dennis Sullivan's major contributions to the field of Dynamical Systems, some of which we attempted to describe here, constitute but one facet of his extraordinary work as a mathematician. There are several other facets. Thus, for his fundamental work in Topology, especially regarding the study of geometric and/or algebraic structures on manifolds, see the article by Shmuel Weinberger in the present volume. In more recent years, Dennis has essentially founded, in collaboration with M. Chas, the sub-field of Topology now known as String Topology. We have heard it said elsewhere that \emph{Dennis Sullivan is a mathematician who has re-invented himself several times\/}, and that seems to us a very accurate  statement. 

We have not included anything about Dennis' recent work on fluid dynamics. Nor have we mentioned any work  on dynamics that Dennis co-wrote with some of his students and/or post-docs, such as the work with Jiang and Morita
 \cite{MR1158758}, his work with Hu \cite{MR1440774}  or his work with Pinto \cite{MR2226493}, nor with many other collaborators from the dynamical systems community. 

In closing, it is important to add that Dennis has always been extremely generous when sharing his ideas with other researchers, as well as in guiding young mathematicians. According to the Math Genealogy Project, he has had so far 40 students, and a total of 155 descendants. Among his students, those who have written a thesis in dynamics under his supervision include Andre de Carvalho, Adam Epstein, Jun Hu, Yunping Jiang, Curt McMullen, Waldemar Paluba, Guiai Peng, Meiyu Su, as well as one of us (EdF). But many more mathematicians, young and old, although not formally his students, have been directly influenced by him. Through his insightful lectures, and his inquisitive quest not merely for results, but for \emph{understanding} Mathematics, Dennis has inspired and will continue to inspire us all.

\acknowledgement{We would like to thank Curt McMullen, Leon Staresinic, Edson Vargas for their useful comments, and especially Dennis Sullivan for explaining to us the origins of his dictionary.}

\bibliography{sullivan-bibl} 
\bibliographystyle{apalike}

\end{document}